\documentclass[12pt,reqno]{amsart}
\usepackage{amssymb,dsfont}
\usepackage{eucal}
\usepackage{amsmath}
\usepackage{amscd}
\usepackage[dvips]{color}
\usepackage{multicol}
\usepackage[all]{xy}           
\usepackage{graphicx}
\usepackage{color}
\usepackage{colordvi}
\usepackage{xspace}
\usepackage{txfonts}
\usepackage{lscape}
\usepackage{tikz} 
\usepackage{color}
\def\supp{\mathrm{supp}}

\numberwithin{equation}{section}
\usepackage[shortlabels]{enumitem}
\usepackage{ifpdf}
\ifpdf
\usepackage[colorlinks,final,hyperindex]{hyperref}
\else
\fi
\usepackage{tikz}
\usepackage[active]{srcltx}

\makeatletter

\numberwithin{equation}{section}
\newtheorem{theo}{Theorem}[section]
\newtheorem{defi}[theo]{Definition}
\newtheorem{coro}[theo]{Corollary}
\newtheorem{lemm}[theo]{Lemma}
\newtheorem{exam}[theo]{Example}
\newtheorem{prop}[theo]{Proposition}

\newtheorem{rema}[theo]{Remark}

\newcommand{\ml}{\mathbf{l}}
\newcommand{\mi}{\mathbf{i}}
\newcommand{\mk}{\mathbf{k}}

\newcommand{\mj}{\mathbf{j}}

\newcommand{\bmi}{\underline{\mathbf{i}}}
\newcommand{\bmj}{\underline{\mathbf{j}}}
\newcommand{\bmk}{\underline{\mathbf{k}}}
\newcommand{\bml}{\underline{\mathbf{l}}}

\begin{document}
	\title[smooth modules]
	{Simple smooth  modules over the superconformal current algebra}

\author{Dong Liu}
\address{Department of Mathematics, Huzhou University, Zhejiang Huzhou, 313000, China}
\email{liudong@zjhu.edu.cn}

\author{Yufeng Pei}
\address{Department of Mathematics, Shanghai Normal University,
Shanghai, 200234, China} \email{pei@shnu.edu.cn}

\author{Limeng Xia}\address{Institute of Applied System Analysis, Jiangsu University, Jiangsu Zhenjiang, 212013, China}\email{xialimeng@ujs.edu.cn}

\author[Zhao]{Kaiming Zhao}
\address{Department of Mathematics, Wilfrid Laurier University, Waterloo, ON, Canada N2L3C5}\email{kzhao@wlu.ca}


\maketitle
		
\begin{abstract}   In this paper, we classify simple smooth modules over the superconformal current algebra $\frak g$. More precisely,
we first classify simple smooth modules over the Heisenberg-Clifford algebra, and then prove that any simple smooth $\frak g$-module is a tensor product of such modules for the super Virasoro algebra and the Heisenberg-Clifford algebra, or an induced module from a simple module over some finite-dimensional solvable Lie superalgebras. As a byproduct,
we provide characterizations for both simple highest weight $\frak g$-modules and simple Whittaker $\frak g$-modules. Additionally, we present several examples of simple smooth $\frak g$-modules that are not tensor product of modules over the super Virasoro algebra and the Heisenberg-Clifford algebra.
\end{abstract}

\tableofcontents


\section{Introduction}	
The super Virasoro algebra ${\frak s}$ and the Heisenberg-Clifford algebra ${\mathfrak {hc}}$ are well-known infinite dimensional Lie superalgebras that have significant applications in various fields of mathematics and physics \cite{KW}. The superconformal current algebra ${\frak g}$ is defined as a Lie superalgebra obtained from the semi-direct product of ${\frak s}$ and ${\mathfrak {hc}}$, as described in Definition \ref{D1}. This particular Lie superalgebra corresponds to 2-dimensional quantum field theories with both chiral and superconformal symmetries (see \cite{KT},\cite{GSW}). Additionally, Balinsky-Novikov superalgebras can also be used to realize the superconformal current algebra ${\frak g}$. 
It was originally introduced by Balinsky for constructing local translation-invariant Lie superalgebras of vector-valued functions on a line (see \cite{B, PB}).  Independently, the superalgebra $\frak g$ was also introduced as an extension of the Beltrami algebra with supersymmetry  (refer to \cite{GSWX}). Marcel, Ovsienko, and Roger studied this superalgebra in their research on generalized Sturm-Liouville operators (refer to \cite{Ma},\cite{MOR}), while it appeared in Gu's work  \cite{Gu} on integrable geodesic flows in the superextension of the Bott-Virasoro group. Recently, there is a study conducted on the superconformal current algebra within supersymmetric warped conformal field theory \cite{CHL}.

Smooth modules for a $\mathbb Z$ or $\frac12\mathbb Z$-graded Lie superalgebra are those in which any vector can be annihilated by a sufficiently large positive part of the Lie superalgebra. Smooth modules are generalizations of highest weight and Whittaker  modules. The study of smooth modules for infinite-dimensional Lie superalgebras with a $\mathbb Z$ or $\frac12\mathbb Z$-gradation is central to Lie theory because these modules are closely related to corresponding vertex operator superalgebras \cite{LL},\cite{Li}. However, classifying all smooth modules for a given Lie superalgebra remains a significant challenge.

Simple smooth modules for the Virasoro algebra and the Neveu-Schwarz algebra were completely determined in \cite{MZ2, LPX1}.
Partial results have been obtained for simple smooth modules of certain Lie (super)algebras in \cite{C, Chris, CG, LvZ, TYZ, XZ}. Among others, simple smooth modules for the Ramond algebra, the twisted Heisenberg-Virasoro algebra, the mirror Heisenberg-Virasoro algebra and the Fermion-Virasoro algebra have been systematically studied.

 In this paper, we   develop some new methods based on   \cite{LPX1, LvZ, TYZ, XL} to provide a  classification of  simple smooth modules for the superconformal current algebra ${\frak g}$.  This seemingly generalization from Lie algebras to Lie supralgebras is indeed nontrivial (overcoming superalgebra difficulties and   combining   results from the super Virasoro algebra, the Fermion-Virasoro algebra and the twisted Heisenberg-Virasoro algebra).
 Our main results are presented below.

\noindent {\bf Main theorem 1}  (Theorem \ref{mainthm})
{\it  Let $S$ be a simple smooth module over the superconformal current algebra ${\frak g}$ with the level $\ell \ne 0$. Then

{\rm(i)} $S\cong H(z)^{{\frak g}}$ where $H$ is a simple  smooth module over the Heisenberg-Clifford superalgebra ${\mathfrak {hc}}$ and $z\in\mathbb{C}$, or

{\rm(ii)} $S$ is an induced ${\frak g}$-module from a  simple smooth ${\frak g}^{(0,-n)}$-module for some $n\in\mathbb Z_+$, or

{\rm(iii)} $S\cong U^{{\frak g}}\otimes H(z)^{{\frak g}}$ where $U$ is a  simple  smooth module over the super Virasoro algebra $\frak s$,   $H$ is a  simple smooth module over the Heisenberg-Clifford superalgebra ${\mathfrak {hc}}$ and $z\in\mathbb{C}$.}

From  {\bf Main theorem 1} we see that simple smooth $\frak g$-modules under some   conditions are tensor products of such modules for the super Virasoro algebra and the Heisenberg-Clifford algebra. Meanwhile,
simple  smooth modules over the Heisenberg-Clifford algebra $\frak{hc}$ will be constructed and classified in Section 3 (Theorem \ref{smooth-hc}), and simple smooth modules over the super Virasoro algebra $\frak{s}$ have been constructed and classified in \cite{LPX1} (Theorem \ref{SVir-modules} below). So {\bf Main theorem 1} shows that simple smooth $\frak g$-modules at nonzero level have been determined.

\noindent {\bf Main theorem 2} (Theorem \ref {MT}) 
{\it Let $S$ be a simple smooth module over the superconformal current algebra ${\frak g}$ of central charge $(c, z)$ at level $\ell=0$, and the central element $H_0$ acts as a scalar $d$ with $d+(n+1)z\ne0$ for any $n\in\mathbb Z^*$.
  Then $S$ is isomorphic to a simple module of the form ${\rm Ind}_{\mathfrak g^{(0, -q)}}^{\mathfrak g}V$, where $V$ is a simple $\mathfrak g^{(0, -q)}$-module for some $q\in\mathbb{N}$}.

\noindent {\bf Main theorem 3} (Theorem \ref {thm-whittaker} and Theorem \ref{whittaker-zero}) 
{\it Suppose that $m\in \mathbb Z_+$, and $\phi_m$ and $\phi'_m$ are given in {\rm Section 6.1.2} with $\phi({\bf c}_2)=z$. 

{\rm (i)} If $\phi_m({\bf c}_3)=\ell\ne0$, then the universal Whittaker $\frak g$-module
$\widetilde{W}_{\phi_m}\cong W_{\phi'_m}^{\frak g}\otimes H(z)^{\frak g}$,  where $W_{\phi'_m}$ is the Whittaker $\frak s$-module and  $H=U({\mathfrak {hc}})w_{\phi_m}$ is the Whittaker $\frak{hc}$-module.  Moreover $\widetilde{W}_{\phi_m}$ is  simple if and only if
$W_{\phi'_m}$ is a simple ${\frak s}$-module.

{\rm (ii)} If $\phi_m({\bf c}_3)=0$, then $\widetilde{W}_{\phi_m}$ is  simple if and only if $\phi_m(H_m)\ne 0$.
  }

It is well known that highest weight modules and Whittaker modules  are important examples of smooth modules.  {\bf Main theorem 3} provides  the necessary and sufficient conditions for   Whittaker modules to be simple in all cases.  Meanwhile by {\bf Main theorems 1, 2}   new characterizations for simple highest weight modules and Whittaker modules are obtained in Section 6 (Theorem \ref{thmmain} and Theorem \ref{thmmain17}).

These simple smooth modules are actually  simple weak modules over the $N=1$ Heisenberg-Virasoro vertex operator superalgebras $\mathcal{V}(c,z,\ell)$ (see \cite{AJR0},\cite{AJR}). Therefore, we are able to classify all weak simple $\mathcal{V}(c,z,\ell)$-modules. It is important to note that certain weak modules induced from simple and smooth ${\frak g}^{(0,-n)}$-modules do not have the form $M_1 \otimes M_2$ as weak modules for a tensor product of two vertex operator superalgebras. Our result in this regard is  interesting because such modules do not exist in the category of ordinary modules for vertex operator superalgebras.

The paper is organized as follows. In Section 2, we recall some related results that will be used throughout the paper.   In Section 3, we obtain all simple smooth modules over the Heisenberg-Clifford superalgebra, 
and
 investigate the tensor product of simple modules over the superconformal current algebra, where we provide some general results. In Section 4, we determine all simple smooth modules for the superconformal current algebra at non-zero level (Theorem \ref{mainthm}). Section 5 determines   simple smooth modules at level zero with a mild condition (Theorem \ref{MT}).  In Section 6, we then apply Theorems  \ref{mainthm} and \ref{MT} to give   characterizations of simple highest weight modules (Theorem \ref{thmmain}) and simple Whittaker modules over ${\mathfrak{g}}$ (Theorem \ref{thm-whittaker}, \ref{thmmain17}). In the end, we provide a few examples of simple smooth ${\mathfrak{g}}$-modules that are not tensor product modules over super Virasoro modules and Heisenberg-Clifford modules.
 
 Throughout  this paper,  $\mathbb Z$, $\mathbb Z^*$,  $\mathbb N$ and $\mathbb Z_+$  denote the set of integers, nonzero integers,     non-negative and positive integers, respectively. Let  $\mathbb C$ and $\mathbb C^*$ denote the sets of complex numbers and nonzero complex numbers, respectively. All vector spaces and Lie superalgebras are assumed to be over $\mathbb C$. We denote the universal algebra of a Lie superalgebra $L$ by $U(L)$.

\section{Notations and preliminaries}
	
In this section, we will review the  notations and established results associated with the superconformal current algebra.

\subsection{Superconformal current algebra}

\begin{defi} \cite{KT}\label{D1} The superconformal current algebra  ${\frak g}$ has a basis $$\left\{L_m, H_m, G_p, F_p, {\bf c}_i \mid i=1,2,3, m\in\mathbb Z, p\in\mathbb Z+\frac12\right\}$$
where
$L_m, H_m, {\bf c}_i\in\frak{g}_{\bar 0}$ and  $G_p, F_p\in\frak{g}_{\bar 1}$,  with the following relations:
\begin{eqnarray*}\label{N=1super}
&& [L_m,L_n]=(m-n)L_{n+m}+\delta_{m+n, 0}{1\over12}(m^3-m){\bf c}_1,\nonumber\\
&& [L_m, H_n]=-nH_{n+m}+\delta_{m+n, 0}(m^2+m){\bf c}_2,\nonumber\\
&&[H_m, H_n]=m\delta_{m+n, 0}{\bf c}_3,\\
&&\label{gd2}[L_m,G_p]=\left(\frac m2-p\right)G_{p+m}, \\ 
&&\label{gd3}[G_p, G_q]=2L_{p+q}+\frac13\left(p^2-\frac14\right)\delta_{p+q, 0}{\bf c}_1,\\
&& [L_m, F_p]=-\left(\frac m2+p\right)F_{m+p},\\ 
&& [F_p, F_q]=\delta_{p+q, 0}{\bf c}_3,\\
&&\label{gd3}[G_p, F_q]=H_{p+q}+(2p+1)\delta_{p+q, 0}{\bf c}_2,\\
&&[H_m, G_p]=mF_{m+p}, \\
&&[H_m, F_p]=0, \ [\frak g, {\bf c}_i]=0
\end{eqnarray*}
for $m,n\in\mathbb Z,\,p, q\in\mathbb Z+\frac12, i=1,2,3$.
\end{defi}

Note that ${\frak g}$ is equipped with a triangular decomposition and $\frac12\mathbb Z$-graded structure:
$
{{\frak g}}={\frak g}^{+}\oplus \frak g_0\oplus {\frak g}^{-},
$
where
\begin{eqnarray*}
	&&{\frak g}^{\pm}=\bigoplus_{n\in\mathbb Z_+, r\in\mathbb N+\frac12}(\mathbb C L_{\pm n}\oplus\mathbb C H_{\pm n}\oplus \mathbb C G_{\pm r}\oplus\mathbb CF_{\pm r}),\\
	&&{\frak g}_0=\mathbb C H_0\oplus\mathbb C L_0\oplus\oplus_{i=1}^3\mathbb C {\bf c}_i.
\end{eqnarray*}
Moreover, \begin{equation}{\frak g}=\bigoplus_{k\in\frac12\mathbb Z}{\frak g}_k\end{equation} is $\frac12\mathbb Z$-graded with 
${\frak g}_i=\mathbb C L_i\oplus\mathbb C H_i, {\frak g}_{r}=\mathbb C G_{r}\oplus\mathbb C F_r\label{grading}$
 for $i\in\mathbb Z^*, r\in\mathbb Z+\frac12$.  
	
	It is worth noting that the center of $\frak g$ is $\mathfrak z:={\rm span}_{\mathbb C}\{H_0, {\bf c}_1,{\bf c}_2,{\bf c}_3\}$.
	Furthermore, it is interesting to observe that ${\frak g}$ contains six important Lie superalgebras:

 \begin{itemize}
\item  the super Virasoro algebra (also called $N=1$ Neveu-Schwarz algebra) $${\frak s}=\text{span}_{\mathbb{C}}\{L_i, G_{i+\frac12},  {\bf c}_1\mid i\in\mathbb Z\},$$ 

\item  the Heisenberg algebra 
$$
\frak h=\text{span}_{\mathbb{C}}\{H_i, {\bf c}_3\mid i\in\mathbb Z\},
$$

\item the Fermion superalgebra 
$$
\frak f=\text{span}_{\mathbb{C}}\{F_{i+\frac12}, {\bf c}_3\mid i\in\mathbb Z\},$$

\item  the Heisenberg-Clifford superalgebra $$\mathfrak {hc}={\rm span}_\mathbb C \{H_i, F_{i+\frac12}, {\bf c}_3\mid i\in\mathbb Z\},$$ 
\item  the twisted Heisenberg-Virasoro algebra $$\frak{hv}=\text{span}_{\mathbb C}\{L_i, H_{i},  {\bf c}_1, {\bf c}_2, {\bf c}_3\mid i\in\mathbb Z\},$$
\item  the Fermion-Virasoro algebra $$\frak{fv}=\text{span}_{\mathbb C}\{L_i, F_{i+\frac12},  {\bf c}_1, {\bf c}_3\mid i\in\mathbb Z\}.$$

 \end{itemize}
 
Set ${\mathfrak {hc}}^\pm={\mathfrak {hc}}\cap\frak g^\pm$ and $\mathfrak s^\pm=\mathfrak s\cap\frak g^\pm$.
	For convenience, we define the following subalgebras of ${\frak g}$. For any $m\in\mathbb N, n\in\mathbb Z$, set
	\begin{equation}\label{Natations}
		\begin{split}
			&{\frak s}^{(m)}= \sum_{i\in\mathbb N}\mathbb C L_{m+i}\oplus\sum_{i\in\mathbb N}\mathbb C G_{m+i+\frac12}\oplus\mathbb C{\bf c}_1,\\
			&{\mathfrak {hc}}^{(m)}=\sum_{i\in\mathbb N}\left(\mathbb C H_{m+i}+\mathbb C F_{m+i+\frac12}\right)\oplus\mathbb C{\bf c}_3\\
		&{\frak g}^{(m, n)}=\frak s^{(m)}\oplus \frak{hc}^{(n)}\oplus\mathbb C{\bf c}_2,\\
		&{\frak g}^{(m, -\infty)}=\frak s^{(m)}\oplus \frak{hc}\oplus\mathbb C{\bf c}_2.
		\end{split}
	\end{equation}

		\begin{defi}
We define a $\mathfrak g$-module $W$ to have central charge $(c, z)$ if ${\bf c}_1$ and ${\bf c}_2$ act on $W$ as complex scalars $c$ and $z$, respectively. Similarly, we say that a $\mathfrak g$-module $W$ has {\bf level} $\ell$, if ${\bf c}_3$ acts on it as the complex scalar $\ell$.

	\end{defi}

\subsection{Some concepts about smooth modules}

In this subsection, we introduce some concepts about smooth modules. 
\begin{defi}
Let $L=\oplus_{i\in\frac12\mathbb Z}{L}_{i}$ be an arbitrary $\frac12\mathbb Z$-graded Lie superalgebra.
An $L$-module $V$ is called  the smooth module if for any $v\in V$ there exists $n\in\frac12\mathbb N$ such that $L_iv=0$,
for $i>n$. The  category of smooth modules over $L$ will be denoted as ${\mathcal R}_{L}$.
\end{defi}

\begin{defi}
Let $\mathfrak{a}$ be a subspace of a Lie superalgebra $L$,  and $V$ be an $L$-module.  We denote
$${\rm Ann}_V(\mathfrak{a})=\{v\in V\mid\mathfrak{a}v=0\}.
$$
\end{defi}

\begin{defi}
Let $L$ be a  Lie superalgebra, $V$ an $L$-module and $x\in L$.
\begin{itemize}
\item[\rm (1)] If for any $v\in V$ there exists $n\in{\mathbb Z}_+$ such that $x^nv=0$,  then we say that the action of  $x$  on $V$ is  locally nilpotent.
\item[\rm (2)] If for any $v\in V$ we have $\mathrm{dim}\,(\sum_{n\in\mathbb N}\mathbb C x^nv)<+\infty$, then   the action of $x$ on $V$ is  said to be locally finite.
\item[\rm (3)] The action of $L$   on $V$  is said to be locally nilpotent if for any $v\in V$ there exists an $n\in{\mathbb Z}_+$ (depending on $v$) such that $x_1x_2\cdots x_nv=0$ for any $x_1,x_2,\cdots, x_n\in L$.
\item[\rm (4)] The action of $L$  on $V$  is said to be locally finite if for any $v\in V$, {\rm dim}\,$U(L)v<\infty$.
\end{itemize}
\end{defi}

\begin{defi}
	A module $M$ over an associative (or Lie) superalgebra $A$ is called strictly simple if it is a simple module over the   algebra $A$ (forgetting the $\mathbb Z_2$-gradation).
\end{defi}

\subsection{Induced modules}

Denote by  $\mathbb{M}$ the set of all infinite vectors of the form $\mi:=(\ldots, i_2, i_1)$ with entries in $\mathbb N$,
satisfying the condition that the number of nonzero entries is finite, and $\mathbb{M}_1:=\{\mathbf{i}\in \mathbb{M}\mid i_k=0, 1, \ \forall k\in\mathbb Z_+\}$.

Let $\mathbf{0}=(\ldots, 0, 0)\in\mathbb{M}$, and
for
$i\in\mathbb Z_+$  let $\epsilon_i$ denote the element $(\ldots,0,1,0,\ldots,0)\in\mathbb{M}$,
where $1$ is
in the $i$'th  position from right.  For any nonzero  $\mi\in\mathbb{M}$, let $\hat{i}$ be the smallest integer $p$ such that $i_p \neq0$ and define  $\mi^\prime=\mi-\epsilon_{\hat{i}}$.
Let $\hat{\hat{i}}$ be the largest integer $q$ such that $i_q \neq0$ and define  $\mi^{\prime\prime}=\mi-\epsilon_{\hat{\hat{i}}}$.

For $\mi, \mj\in\mathbb{M}, \mk, \ml\in\mathbb{M}_1$, we denote
$$\ell(\mi)=\sum_{j\ge 1}i_j,$$
$$L^{\mi}H^{\mj}G^{\mk}F^{\ml} =\ldots L_{-2}^{i_2} L_{-1}^{i_1}\ldots H_{-2}^{j_2}H_{-1}^{j_1}\ldots G_{-2+\frac{1}{2}}^{k_2} G_{-1+\frac{1}{2}}^{k_1}\ldots F_{-2+\frac{1}{2}}^{l_2} F_{-1+\frac{1}{2}}^{l_1}\in U({\frak g}_{-})$$
and
\begin{equation} {\rm w}(\mi, \mj, \mk, \ml)=\sum_{n\in\mathbb Z_+}i_n\cdot n+\sum_{n\in\mathbb Z_+}j_n\cdot n+\sum_{n\in\mathbb Z_+}k_n\cdot (n-\frac12)+\sum_{n\in\mathbb Z_+}l_n\cdot (n-\frac12),\label{length}\end{equation} which is called the length of $(\mi, \mj, \mk, \ml)$ (or the length of $L^{\mi}H^{\mj}G^{\mk}F^{\ml}$).
\begin{equation} {\rm w}(\mi, \mk)=\sum_{n\in\mathbb Z_+}i_n\cdot n+\sum_{n\in\mathbb Z_+}k_n\cdot (n-\frac12), \end{equation} which is called the length of $(\mi,\mk)$ (or the length of $L^{\mi}G^{\mk}$ or the length of $H^{\mi}F^{\mk}$).

Denote by $<$ the  {\em  lexicographical  total order}  on  $\mathbb{M}$,  defined as follows: for any $\mi,\mj\in\mathbb{M}$, set $$\mi< \mathbf{j}\ \Leftrightarrow \ \mathrm{ there\ exists} \ r\in\mathbb Z_+ \ \mathrm{such \ that} \ i_r<j_r\ \mathrm{and} \ i_s=j_s,\ \forall\; s>r.$$

 Denote by $\prec$ the  {\em reverse  lexicographical  total order}  on  $\mathbb{M}$,  defined as follows: for any $\mi,\mj\in\mathbb{M}$, set $$\mi\prec \mathbf{j}\ \Leftrightarrow \ \mathrm{ there\ exists} \ r\in\mathbb Z_+ \ \mathrm{such \ that} \ i_r<j_r\ \mathrm{and} \ i_s=j_s,\ \forall\; 1\le s<r.$$

Now we introduce two total orders "$<$" and "$\prec$" on $\mathbb{M}\times\mathbb{M}_1$:

\begin{defi}\label{order}  
Denoted by $<$:
$(\mj,\ml) < (\mj_1,\ml_1)$ if and only if one of the following conditions is satisfied:


{\rm (1)} $\mj <\mj_1$;

{\rm (2)} $\ \mj=\mj_1\
\mathrm{and }\ \ml < \ml_1$  for all $\mj, \mj_1\in\mathbb M, \mathbf{l}, \ml_1\in\mathbb{M}_1$.
\end{defi}

\begin{defi}\label{order}  
Denoted by $\prec$:
$(\mi,\mk) \prec (\mi_1,\mk_1)$ if and only if one of the following conditions is satisfied:

{\rm (1)} ${\rm w}(\mi, \mk)<\rm {w}(\mi_1, \mk_1)$;

{\rm (2)} ${\rm w}(\mi, \mk)=\rm {w}(\mi_1, \mk_1)$ and $\mk\prec\mk_1$;

{\rm (3)} ${\rm w}(\mi, \mk)=\rm {w}(\mi_1, \mk_1)$, $\mk=\mk_1$, $\ell(\mi)<\ell(\mi_1)$;

{\rm (4)} ${\rm w}(\mi, \mk)=\rm {w}(\mi_1, \mk_1)$, $\mk=\mk_1$, $\ell(\mi)=\ell(\mi_1)$ and $\mi\prec \mi_1$, 
  for all $\mi, \mi_1\in\mathbb M, \mathbf{k}, \mk_1\in\mathbb{M}_1$.
\end{defi}

\begin{defi}\label{order}    

Now we can induce a  principal total order on $\mathbb{M}^2\times\mathbb{M}_1^2$, still denoted by $\prec$:
$(\mi, \mj, \mk, \ml) \prec (\mi_1, \mj_1, \mk_1, \ml_1) $ if and only if one of the following conditions is satisfied:

{\rm (1)} ${\rm w}(\mi, \mj, \mk, \ml)<{\rm w}(\mi_1, \mj_1, \mk_1, \ml_1)$;

{\rm (2)} ${\rm w}(\mi, \mj, \mk, \ml)={\rm w}(\mi_1, \mj_1, \mk_1, \ml_1)$ and $\mk \prec \mathbf{k}_1$;

{\rm (3)} ${\rm w}(\mi, \mj, \mk, \ml)={\rm w}(\mi_1, \mj_1, \mk_1, \ml_1)$ and $\mk =\mathbf{k}_1$ and $\mi \prec \mathbf{i}_1$;

{\rm (4)} ${\rm w}(\mi, \mj, \mk, \ml)={\rm w}(\mi_1, \mj_1, \mk_1, \ml_1)$, $(\mi,\mk) =(\mi_1,\mathbf{k}_1)$ and $(\mj, \ml) < (\mj_1, \ml_1)$.
\end{defi}

For $q\in\mathbb N$ or $q=\infty$, let $V$ be a $\frak g^{(0, -q)}$-module.
According to the $\mathrm{PBW}$ Theorem, every element of $\mathrm{Ind}_{\frak g^{(0, -q)}}^{\frak g}V$ can be uniquely written in the
following form
\begin{equation}\label{def2.1}
\sum_{\mi, \mj\in\mathbb{M},\mk, \ml\in\mathbb{M}_1}L^{\mi}H^{\mj}G^{\mk}F^{\ml} v_{\mi, \mj, \mk,\ml},
\end{equation}
where all  $v_{\mi, \mj, \mk,\ml}\in V$ and only finitely many of them are nonzero. 

\begin{defi}
For any $v\in\mathrm{Ind}_{\frak g^{(0, -q)}}^{\frak g}V$ as in  \eqref{def2.1}, we denote by $\mathrm{supp}(v)$ the set of all $(\mi, \mj, \mk,\ml)\in \mathbb{M}^2\times\mathbb{M}_1^2$  such that $v_{\mi, \mj, \mk,\ml}\neq0$. For a nonzero $v\in \mathrm{Ind}_{\frak g^{(0, -q)}}^{\frak g}V$, we write $\mathrm{deg}(v)$  the maximal element (with respect to the principal total order on $\mathbb{M}^2\times\mathbb{M}_1^2$) in $\mathrm{supp}(v)$, called the {\it  degree} of $v$. 
 \end{defi}

 For later use, we also define $\mathrm{supp}_{\frak s}(v)$ to be the set of all $(\mi,\mk)\in \mathbb{M}\times\mathbb{M}_1$  with $v_{\mi, \mj, \mk,\ml}\neq0$, $\mathrm{supp}_{\frak{hc}}(v)$ to be the set of all $(\mj, \ml)\in \mathbb{M}\times\mathbb{M}_1$  with $v_{\mi, \mj, \mk,\ml}\neq0$, 
 $\mathrm{deg}_{\frak s}(v)$ to be  the maximal element (with respect to the principal total order $\prec$ on $\mathbb{M}\times\mathbb{M}_1$) in $\mathrm{supp}(v)$,  $\mathrm{deg}_{\frak {hc}}(v)$ to be  the maximal element (with respect to the principal total order $<$ on $\mathbb{M}\times\mathbb{M}_1$) in $\mathrm{supp}_{\frak {hc}}(v)$, and denote by $|v|_{\frak s}=\text w({\rm deg}_{\frak s}(v))$,  $|v|_{\frak {hc}}=\text w({\rm deg}_{\frak {hc}}(v))$. Note that here and later we make the convention 
that  $\mathrm{deg}(v)$, $\mathrm{deg}_{\frak s}(v)$ and $\mathrm{deg}_{\frak {hc}}(v)$ only for  $v\neq0$. 
Now we give a lemma for later convenience.

\begin{lemm}\label{homo}
For any $k, l\in\mathbb Z_+, k\ge l$,  $x\in U(\frak g^-\oplus\frak g^0)_{-k}$ and $y\in \frak g_{l}$, we have 
$[y, x]\in U(\frak g^-\oplus\frak g^0)_{-k+l}+U(\frak g^-\oplus\frak g^0)\frak g_+$, where the grading of $U(\frak g^-\oplus\frak g^0)_{-k}$ is according to the adjoint action of $L_0$.
\end{lemm}
\begin{proof}
It follows by  direct calculations.
\end{proof}

\section{Simple smooth modules over the Heisenberg-Clifford algebra}

In this section we shall consider all simple smooth ${\mathfrak {hc}}$-modules with nonzero level.

\subsection{Tensor product of simple modules}

We need the following result on simple tensor product modules.
\begin{lemm}\cite[Lemma 2.2]{XL}\label{tensor-simple}
Let $A, A'$ be unital associative superalgebras, and $M, M'$ be $A, A'$ modules, respectively.
If $A'$ has a countable basis and $M'$ is strictly simple, then

$(1)$ $M\otimes M'$ is a simple $A\otimes A'$-module if and only if $M$ is a simple $A$-module;

$(2)$ if $V$ is a simple $A\otimes A'$-module containing a strictly simple $\mathbb C\otimes A'$-submodule $M'$, then $V\cong M\otimes M'$ for some simple $A$-module
$M$.
\end{lemm}

\begin{lemm}\label{induced}Let $W$ be a module over a Lie superalgebra $L$, $S$ be a subalgebra of $L$, and $B$ be an $S$-module.
Then the $L$-module homomorphism $\varphi:
{\rm Ind}_{S}^L(B\otimes W)\rightarrow 
{\rm Ind}_{S}^L(B)\otimes W$ induced from the inclusion map $B\otimes W\rightarrow {\rm Ind}_{S}^L(B)\otimes W$ is an $L$-module
isomorphism.
\end{lemm}
\begin{proof}
It is essentially the same as that of \cite[Lemma 8]{LvZ} although we have modules over Lie superalgebras here.
\end{proof}



Let ${\mathfrak g}={\mathfrak a}\ltimes{\mathfrak b}$ be a Lie superalgebra where ${\mathfrak a}$ is a subalgebra  of ${\mathfrak g}$ and ${\mathfrak b}$ is an ideal of ${\mathfrak g}$. Let $M$ be a  ${\mathfrak g}$-module with a   ${\mathfrak b}$-submodule $H$ so that the ${\mathfrak b}$-submodule structure on $H$ can be extended to a ${\mathfrak g}$-module structure on $H$. We denote this ${\mathfrak g}$-module by $H^{\mathfrak g}$. For any ${\mathfrak a}$-module   $U$, we can make it into a ${\mathfrak g}$-module by ${\mathfrak b}U=0$. We denote this ${\mathfrak g}$-module by $U^{\mathfrak g}$. Then by Lemma \ref{tensor-simple} we have

\begin{coro}\label{simple-app} Let ${\mathfrak g}={\mathfrak a}\ltimes{\mathfrak b}$ be a countable dimensional Lie superalgebra.
Let $M$ be a  simple ${\mathfrak g}$-module
with a strictly simple  ${\mathfrak b}$-submodule $H$ so that an $H^{\mathfrak g}$ exists.
Then $M\cong U^{\mathfrak{g}}\otimes H^{\mathfrak{g}}$ as $\mathfrak{g}$-modules for some  simple  ${\mathfrak a}$-module $U$.
\end{coro}

For any $z\in\mathbb C$, $H\in{\mathcal R}_{\mathfrak {hc}}$ with the action of ${\bf c}_3$ as a nonzero scalar $\ell$,
similar to (4.2), (4.3), (4.4a), (4.4b) in \cite{KT} where they only considered highest weight modules, here  we can also give  a ${{\frak g}}$-module structure on $H$
(denoted by $H(z)^{{\frak g}}$) via the following map
\begin{align}
&L_n\mapsto \bar{L}_n=\frac{1}{2\ell}\sum_{k\in \mathbb Z}:H_{k} H_{-k+n}:-\frac{(n+1)z}{\ell}H_n
-\frac{1}{2\ell}\sum_{r\in \mathbb Z+\frac12}(r+\frac12):F_{r} F_{-r+n}:,\nonumber\\
&G_r\mapsto \bar{G}_r=\frac{1}{\ell}\sum_{k\in \mathbb Z}H_{k} F_{-k+r}-\frac{(2r+1)z}{\ell}F_r, \
H_n\mapsto H_n,\ F_r\mapsto F_r, \nonumber\\
&{\bf c}_1\mapsto\frac32-\frac{12z^2}{\ell}, \  {\bf c}_2\mapsto z, \ {\bf c}_3\mapsto \ell\label{rep2}
\end{align}
for all $n\in\mathbb Z, r\in \mathbb Z+\frac12$, where the normal order is defined as $$:H_iH_j:=\left\{
                                       \begin{array}{cc}
                                        H_i H_j,&\mbox{if}\,\, i<j, \\
                                       H_j H_i, & \,\,\mbox{otherwise.}\\
                                     \end{array}
                                    \right.$$
                                    $$:F_rF_s:=\left\{
                                       \begin{array}{cc}
                                        F_r F_s,&\mbox{if}\,\, r<s, \\
                                       -F_s F_r, & \,\,\mbox{otherwise.}\\
                                     \end{array}
                                    \right.$$

Applying the above corollary to $\mathfrak{g}={\frak s}\ltimes {\mathfrak {hc}}$, we have the following results.

\begin{coro}\label{tensor} Let $V$ be a simple smooth ${\mathfrak{g}}$-module with central charge $(c, z)$ at nonzero level. Then $V\cong  U^{\mathfrak{g}}\otimes H(z)^{\mathfrak{g}}$ as a $\mathfrak{g}$-module for some simple ${\frak s}$-module $U\in {\mathcal R}_{\frak s}$ and some  simple module $H\in{\mathcal R}_{\mathfrak {hc}}$  if and only if $ V$ contains
 a simple ${\mathfrak {hc}}$-submodule $H\in{\mathcal R}_{\mathfrak {hc}}$.
 \end{coro}
\begin{proof} ($\Leftarrow$). This follows from Corollary \ref{simple-app} and the fact that any simple smooth $\frak{hc}$-module $H$ is  strictly simple (see Section 3 of \cite{XZ}, also see Proposition \ref{smooth-hc} below).

($\Rightarrow$). This follows from the fact that $u \otimes H(z)^{\mathfrak{g}}$ is a simple ${\mathfrak {hc}}$-submodule of $U^{\mathfrak{g}} \otimes H(z)^{\mathfrak{g}}$ for any nonzero $u\in U$.
\end{proof}

\subsection{Simple modules in ${\mathcal R}_{{\mathfrak {hc}}}$}

For any $m\in\mathbb Z_+$, we define $$\mathfrak {hc}_{[m]}={\rm span}\{H_i, F_{j+\frac12}, 
{\bf c}_3|-m\le i\le m, -m\le j\le m-1\}.$$

\begin{theo}\label{smooth-hc}
Let $B$, $B'$ be simple modules over $\mathfrak {hc}_{[m]}$ for some $m\in \mathbb Z_+$  with  nonzero action of ${\bf c}_3$.

{\rm (a)} The ${\mathfrak {hc}}$-module ${\rm Ind}_{{\mathfrak {hc}}^{(-m)}}^{\mathfrak {hc}} B$ is simple, where $B$ is
regarded as ${\mathfrak {hc}}^{(-m)}$-module by $H_i B=0=F_{i+\frac12} B, \forall i>m$. Moreover, all
nontrivial simple   modules in ${\mathcal R}_{\mathfrak {hc}}$ can be obtained in
this way.

{\rm (b)} As ${\mathfrak {hc}}$-modules, ${\rm Ind}_{{\mathfrak {hc}}^{(-m)}}^{\mathfrak {hc}} B\cong {\rm Ind}_{{\mathfrak {hc}}^{(-m)}}^{\mathfrak {hc}} B'$
if and only if $B\cong B'$ as $\mathfrak {hc}_{[m]}$-modules.

{\rm (c)}  Let $V$ be a simple smooth ${\mathfrak {hc}}$-module with  nonzero action of ${\bf c}_3$, then  $V\cong H\otimes K $ for some simple
smooth $\frak h$-module $H$ and some simple smooth $\frak f$-module $K$, where the actions are as follows:
$$x(u\otimes w)=\left\{\begin{matrix}
xu\otimes w & \text{ if }x\in \frak{h}, \\
u\otimes xw & \text{ if }x\in \frak{f}. \\
\end{matrix}\right.$$

\end{theo}

\begin{proof}
We define the new Lie superalgebra $\frak f'=\{F_r, {\bf c}'_3|r\in \mathbb{Z}\}$ with brackets:
$$ 
  [F_m, F_n]=\delta_{m+n, 0}{\bf c}'_3,$$
and the subalgebras $\frak f'_{[m]}=\{F_r, {\bf c}'_3|-m+\frac12\le r\le m-\frac12\}$, 
$\frak{h}_{[m]}={\rm span}\{H_i , 
{\bf c}_3|-m\le i\le m\}$.
Then 
$$\mathfrak {hc}_{[m]} =(\frak h_{[m]}\oplus\frak f'_{[m]})/\langle {\bf c}_3-{\bf c}_3'\rangle.$$
Now $B, B'$ can be considered as modules over the superalgebra $(\frak h_{[m]}\oplus\frak f'_{[m]})$. From Lemma 3.8
in \cite{XZ}, we know that there is a strictly simple $\frak f'_{[m]}$-module $B_2$ contained in $B$. From Corollary \ref{simple-app}, we know that there is a simple $\frak h_{[m]}$-module $B_1$ such that $B=B_1\otimes B_2$ as module over the superalgebra $(\frak h_{[m]}\oplus\frak f'_{[m]})$.

(a) Using PBW Theorem and nonzero action of ${\bf c}_3$, we can easily
prove that the ${\mathfrak {hc}}$-module ${\rm Ind}_{{{\mathfrak {hc}}}_{[m]}}^{\mathfrak {hc}} B ={\rm Ind}_{{{\frak h}}_{[m]}}^{\frak h} B_1\otimes {\rm Ind}_{{{\frak f}}'_{[m]}}^{\frak h} B_2$
is simple.

Now suppose that $V\in {\mathcal R}_{\mathfrak {hc}}$ is simple and nontrivial. From Lemma 3.9
in \cite{XZ}, we know that there is a simple $\frak f'$-submodule $W_2$   in $V$.
From Corollary \ref{simple-app}, we know that there is a simple $\frak h$-module $W_1$ such that $V=W_1\otimes W_2$ as module over the superalgebra $(\frak h\oplus\frak f')$. Combining this with  Lemma 3.6
in \cite{XZ} and  Proposition 9
in \cite{LvZ} we obtain the result.

(b) Noting that $B$ and $ B'$ are the socles of ${\mathfrak {hc}}^{(-m)}$-modules
${\rm Ind}_{{\mathfrak {hc}}^{(-m)}}^{\mathfrak {hc}} B$ and ${\rm Ind}_{{\mathfrak {hc}}^{(-m)}}^{\mathfrak {hc}} B'$ respectively, we see
that, if ${\rm Ind}_{{\mathfrak {hc}}^{(-m)}}^{\mathfrak {hc}} B\cong {\rm Ind}_{{\mathfrak {hc}}^{(-m)}}^{\mathfrak {hc}} B'$  as
${\mathfrak {hc}}^{(-m)}$-modules, hence  $B\cong B'$ as $\mathfrak {hc}_{[m]}$-modules. The
converse is trivial.

(c) 
Note that  $\mathfrak {hc}=(\frak h\oplus\frak f')/\langle {\bf c}_3-{\bf c}_3'\rangle.$ Then $V$ is naturally considered as a $(\frak h\oplus\frak f')$-module. By   Lemma 3.9 in \cite{XZ}, there is a strictly simple smooth $\frak f'$-submodule $K$ in $V$. Using Corollary \ref{simple-app}, we obtain the statement.  
\end{proof}

Note that, simple
smooth $\frak h$-modules were classified in \cite{LvZ}, while 
simple smooth $\frak f$-modules were classified in \cite{XZ}.
Combining with the above proposition,  we have   all simple smooth modules over   the Heisenberg-Clifford superalgebra $\mathfrak {hc}$.

\begin{exam} Take $m=1$, make $\mathbb C v$ into a module over $\mathfrak{b}=\mathbb C\{H_0, H_{-1}, F_{-\frac12}\}+\mathbb C {\bf c}_3$ by  $H_0v=dv, H_{-1}v=av, F_{-\frac12}v=0,  {\bf c}_3v=\ell v$ for $a, d, \ell\in\mathbb C$ with $\ell\ne0$. We have the simple ${\mathfrak {hc}}_{[1]}$-module $B={\rm Ind}^{{\mathfrak {hc}}_{[1]}}_\mathfrak{b}\mathbb C v$. Then we obtain the simple   ${\mathfrak {hc}}$-module ${\rm Ind}_{{{\mathfrak {hc}}}^{(-1)}}^{\mathfrak {hc}} B$. In this manner similar to Sect.3.1 in \cite{BBFK}, one can construct a lot of simple nonweight modules in ${\mathcal R}_{{\mathfrak {hc}}}$.
\end{exam}

Now we introduce a result on  Whittaker modules over ${\mathfrak {hc}}$.  Suppose that $\phi: {\mathfrak {hc}}^{(0)}\rightarrow \mathbb C$ is a homomorphism of Lie superalgebras. It follows that $\phi(F_{r})=0$ for all $r>0$. Then $\mathbb C w_{\phi}$ becomes a one-dimensional 
${\mathfrak {hc}}^{(0)}$-module defined by $x w_{\phi}=\phi(x)w_{\phi}$ for all
$x\in {\mathfrak {hc}}^{(0)}$. The induced ${\mathfrak {hc}}$-module
$W_{\phi}={\rm Ind}_{{\mathfrak {hc}}^{(0)}}^{\mathfrak {hc}} \mathbb C w_{\phi}$ is called a
Whittaker module with respect to $\phi$. Note that the  ${\mathfrak {hc}}$-module $W_{\phi}$ is not necessarily a smooth module.

\begin{lemm} \label{H-whittaker}The ${\mathfrak {hc}}$-module $W_{\phi}$ is  simple  if and only if $\phi({\bf c}_3)\ne 0$.\end{lemm}
\begin{proof}
	It follows by direct calculations.
\end{proof}

\subsection{Simple modules  in $ \mathcal{R}_{\mathfrak{s}}$}

We first recall from \cite{LPX1} the classification for  simple smooth ${\frak s}$-modules.

\begin{theo}\cite{LPX1} \label{SVir-modules} Any  simple smooth ${\frak s}$-module is a highest weight module, or isomorphic to $ {\rm Ind}_{{\frak s}^{(0)}}^{{\frak s}}V$ for a  simple  ${\frak s}^{(0)}$-module $V$ such that for some $k\in {\mathbb Z}_+$,
	
	{\rm (i)} $L_k$ acts injectively on $V$;
	
	{\rm (ii)} $L_iV=0$ for all $i>k\, ($in this case $G_{i-\frac12}V=0$ for all $i>k)$.
\end{theo}

For later convenience, now we establish some basic results about the Verma modules over the super Virasoro algebra.

For the Verma module $M_{{\frak s}}(c,h)$ over the super Virasoro algebra ${\frak s}$,
it is well-known from \cite{A} that there exist two homogeneous elements $P_1, P_2\in U({\frak s}^-){\frak s}^-$
such that $ U({\frak s}^-)P_1w_1+ U({\frak s}^-)P_2w_1$ is the unique maximal proper ${\frak s}$-submodule of $M_{{\frak s}}(c,h)$,
where $P_1, P_2$ are allowed to  be zero and $w_1$ is the
highest weight vector in $M_{{\frak s}}(c,h)$.

\begin{lemm}\label{singular-vector}
Let $k=0,-1$. Suppose $M$ is an ${\frak s}^{(k)}$-module on which $L_0$ and $\bf c_1$ act as multiplication by  given scalars $h$ and $c$ respectively. Then there exists a unique maximal submodule $N$ of ${\rm Ind}^{{\frak s}}_{{\frak s}^{(k)}}M$ with $N\cap M=0$. More precisely, $N$ is generated by $P_1M$ and $P_2M$, i.e., $N= U({\frak s}^-)(P_1M+P_2M)$.
\end{lemm}

\begin{proof} We will follow the proof in \cite[Lemma 4.8]{TYZ}.
Note that $L_0$ acts diagonalizably on ${\rm Ind}^{{\frak s}}_{{\frak s}^{(k)}}M$ and its submodules, and  $$M=\{u\in{\rm Ind}^{{\frak s}}_{{\frak s}^{(k)}}M\mid L_0u=\lambda u\},$$
i.e., $M$ is the highest weight space of ${\rm Ind}^{{\frak s}}_{{\frak s}^{(k)}}M$.
 Let $N$ be the sum of all ${\frak s}$-submodules of ${\rm Ind}^{{\frak s}}_{{\frak s}^{(k)}}M$ which intersect with $M$ trivially. Then $N$ is the desired unique maximal ${\frak s}$-submodule of ${\rm Ind}^{{\frak s}}_{{\frak s}^{(k)}}M$ with $N\cap M=0$.

Let $N^{\prime}$ be the ${\frak s}$-submodule generated by $P_1M$ and $P_2M$, i.e., $N^{\prime}= U({\frak s}^-)(P_1M+P_2M)$. Then $N^{\prime}\cap M=0$. Hence, $N^{\prime}\subseteq N$. Suppose there exists a proper submodule $U$ of ${\rm Ind}^{{\frak s}}_{{\frak s}^{(k)}}M$ such that $U\subset N$ and  $U\not\subset N^{\prime}$.  Choose a nonzero homogeneous $v=\sum _{i=1}^ru_iv_i\in U\setminus N^{\prime}$, where $u_i\in  U({\frak s}^-)$ and $v_1,...v_r\in M$ are linearly independent. Note that all $u_i$ have the same weight. Then some $u_iv_i\notin N^{\prime}$, say $u_1v_1\notin N^{\prime}$.
There is a  homogeneous $u\in  U({\frak s})$ such that $uu_1v_1=v_1$.
Noting that all $uu_i$ has weight $0$ and each $U(\mathfrak s)v_i$ is a highest weight $\mathfrak s$-module, so $uu_iv_i\in \mathbb C v_i$. Thus $uv\in  U\cap M\subset  N\cap M=0$, which is impossible.  This implies that $N\subseteq N^{\prime}$. Hence, $N=N^{\prime}$, as desired.
\end{proof}

\begin{lemm}\label{free-L0}
Let $M$ be an ${\frak s}^{(0)}$-module on which  ${\frak s}^{+}$ acts  trivially and   $\bf c_1$ acts as multiplication by  ascalar  $c$. If any finitely generated $\mathbb C[L_0]$-submodule of $M$ is a free $\mathbb C[L_0]$-module, then any nonzero $\frak s$-submodule  of ${\rm Ind}^{\frak s}_{{\frak s}^{(0)}}M$  intersects with $M$ non-trivially.
\end{lemm}
\begin{proof}  We will follow the proof in  \cite[Lemma 4.10]{TYZ}.
Let $V$ be a nonzero submodule of ${\rm Ind}^{\frak s}_{{\frak s}^{(0)}}M$.
Take a nonzero $u\in V$ and  $u\in V\backslash M$. Write
$u=\sum_{i=1}^n a_iu_i$ where $a_i\in  U({\frak s}^-\oplus \mathbb CL_0)$, $u_i\in M$. Since $M_1:=\sum_{1\le i\le n}\mathbb C[L_0]u_i$ (an ${\frak s}^{(0)}$-submodule of $M$ ) is a finitely generated $\mathbb C[L_0]$-module, we see that $M_1$ is a free module over $\mathbb C[L_0]$ by the assumption. Without loss of generality, we may assume that $M_1=\oplus_{1\le i\le n}\mathbb C [L_0]u_i$ with basis  $u_1,\cdots,u_n$ over $\mathbb C[L_0]$. Note that each $a_i$ can be expressed as a sum of eigenvectors of ${\rm ad}\, L_0$ for $1\leq i\leq n$. Assume that $a_1$ has a maximal eigenvalue among all $a_i$ for $1\leq i\leq n$. Then $a_1u_1\notin M$. For any  $\lambda\in\mathbb C$, let $M_1(\lambda)$ be  the $\mathbb C[L_0]$-submodule  of $M_1$  generated by  $u_2, u_3,\cdots, u_n, L_0u_1-\lambda u_1$. Then $M_1/M_1(\lambda)$ is a one-dimensional
${\frak s}^{(0)}$-module with $L_0(u_1+M_1(\lambda))=\lambda u_1+M_1(\lambda)$. By the Verma module theory for the super Virasoro algebra, we know that there exists    $0\ne \lambda_0\in \mathbb C$ such that the corresponding Verma module $U={\rm Ind}_{{\frak s}^{(0)}}^{\frak s}(M_1/M_1(\lambda_0))$
is irreducible (see \cite{A,IK1}). We know that $u=a_1u_1\ne0$ in $U$.
Hence we can find a  homogeneous $w\in  U({\frak s}^+)$ such that $wa_1u_1=f_1(L_0)u_1$ in ${\rm Ind}^{\frak s}_{{\frak s}^{(0)}}M$, where $0\neq f_1(L_0)\in\mathbb C[L_0]$.
So $wu=\sum_{i=1}^n wa_iu_i=\sum_{i=1}^n f_i(L_0)u_i $ for $f_i(L_0)\in\mathbb C[L_0]$, $1\leq i\leq n$. Therefore, $0\ne wu\in V\cap M_1\subset V\cap M,$ as desired.
\end{proof}

\section{Simple smooth ${\frak g}$-modules with nonzero level}\label{char}

In this section we will determine all simple  smooth ${\frak g}$-modules $M$ of nonerzo level.  

For a given simple smooth ${\frak g}$-module $M$ with level $\ell \not=0$ and central charge  $(c, z)$, From \cite[Lemma 3.9]{XZ}, we know that    ${\rm Ann}_M(\frak f^+)\ne0$. Then there is an $n\in \mathbb N$ such that $$M(n)={\text{Ann}}_M({\rm span}_{\mathbb C}\{H_i |i\ge n\})\cap {\rm Ann}_M(\frak f^+)\ne0.$$ Let 
$n_M=\min\{n\in \mathbb Z:M(n)\ne0\},$ and 
$$ \aligned M_0&={\text{Ann}}_M({\rm span}_{\mathbb C}\{H_i, F_{i-\frac12} |i\ge n_M\}), \text{ if }n_M>0,\\
 M_0&={\text{Ann}}_M({\rm span}_{\mathbb C}\{H_i, F_{ i+\frac12} |i\ge 0\}), \text{ if }n_M=0.\endaligned$$

\begin{lemm}\label{lemma-n} Let $M$ be a
simple smooth $\frak g$-module  with level $\ell \not=0$. Then the following statements hold.
\begin{enumerate}
\item[\rm(i)] $n_M\in\mathbb N$, and $H_{n_M-1}$ acts injectively on $M_0$.
\item[\rm(ii)] $M_0$ is a nonzero $\frak g^{(0,-(n_M-1))}$-module, and is invariant under the action of the operators $\bar L_i, \bar G_{i+\frac12}$  defined in \eqref{rep2} for $i\in\mathbb N$.
\end{enumerate}
\end{lemm}
\begin{proof}
(i) Assume that $n_M<0$. Take any nonzero $v\in M_0$, we then have
$$v=\frac{1}{\ell}[H_{1},H_{-1}]v=0,$$ a contradiction.  Hence, $n_M\in\mathbb N$.
The definition of $n_M$ means that $H_{n_M-1}$ acts injectively on $M_0$.

(ii) It is obvious that $M_0\neq 0$ by definition. For any $w\in M_0$, $i, k\in\mathbb N$, it is clear that
 $L_iw, G_{i+\frac12}w, H_{i-(n_M-1)}w, F_{i-(n_M-1)+\frac12}w\in M_0$. So $M_0$ is a nonzero ${\frak g}^{(0,-(n_M-1))}$-module.

For $i, n\in\mathbb N$, $w\in M_0$, noticing $n_M\ge 0$ by (i), it follows from (\ref{rep2}) that
\begin{eqnarray*}
&&H_{i+n_M}\bar L_nw=\bar L_nH_{i+n_M}w+(i+n_M)H_{n+i+n_M}w=0,\\
&&F_{i+n_M\pm \frac12}\bar L_nw=\bar L_nF_{i+n_M\pm\frac12}w+(\frac{n}2+i+n_M\pm \frac12)F_{n+i+n_M\pm \frac12}w=0,
\end{eqnarray*}
where sign $``+"$ corresponds to $n_M=0$.
This implies that $\bar L_nw\in M_0$ for $n\in\mathbb N$. Similarly,  $\bar G_{n+\frac12}w\in M_0$ for $n\in\mathbb N$.  That is, $M_0$ is invariant under the action of the operators $\bar  L_n, \bar G_{n+\frac12}$ for $n\in\mathbb N$.
\end{proof}

\begin{prop}\label{prop-n=01} Let $M$ be a simple smooth $\frak g$-module  with level $\ell \not=0$  and central charge  $(c, z)$.
 If $n_M=0, 1$, then  $M\cong U^{\frak g}\otimes H(z)^{\frak g}$ as $\frak g$-modules for some  simple modules $U\in \mathcal{R}_{{\frak s}}$ and
 $H\in \mathcal{R}_{\mathfrak {hc}}$.
 \end{prop}
\begin{proof}
Since  $n_M=0, 1$, we  take any nonzero $v\in M_0$. Then $\mathbb C v$ is an  ${\mathfrak {hc}}^{(0)}$-module (the action of $H_0$ is a scalar multiplication since it is in the center of $\frak g$). Let $H=U({\mathfrak {hc}})v$, the ${\mathfrak {hc}}$-submodule of $M$ generated by $v$. By some direct calculations we see that $H$ is a simple ${\mathfrak {hc}}$-module.   Then the desired assertion follows directly from  Corollary \ref{tensor}.
\end{proof}

Next we assume that $n_M\ge 2$.

We define the operators $L_n'=L_n-\bar L_n$ and $G_{n+\frac12}'=G_{n+\frac12}-\bar G_{n+\frac12}$ on $M$ for $n\in\mathbb Z$.  Since $M$ is a smooth $ {\frak g}$-module, then $L_n'$ is well-defined for any $n\in\mathbb Z$. 

Since 
\begin{equation}\label{LHG}\aligned
&[L_n,H_{k}]=[\bar L_n,H_{k}]=-kH_{n+k}+\delta_{n+k,0}(n^2+n){\bf c}_2,  \\
 &[G_{n+\frac12}, H_{k}]=[\bar G_{n+\frac12}, H_{k}]=-F_{n+k+\frac12}, \\  
 &[L_n, F_{r}]=[\bar L_n, F_{r}]=-(\frac n2+r)F_{n+r}, \\ 
 &[G_p, F_{r}]=[\bar G_p, F_r]=H_{p+q}+(2p+1)\delta_{p+q, 0}{\bf c}_2,\endaligned
\end{equation}
 we have
\begin{equation}\label{N=1super11}
\aligned
& [\bar L_m, L_n]=[\bar L_m, \bar L_n],\  [\bar L_m,G_p]=[\bar L_m, \bar G_p], \\ 
& [ L_m,\bar G_p]=[\bar L_m, \bar G_p],\  [\bar G_p, G_q]=[\bar G_p, \bar G_q].\endaligned
\end{equation}
By \eqref{rep2} and \eqref{N=1super11}, we have
\begin{eqnarray}\label{vir-bracket'}
&&[L_m',L_n']=(m-n)L_{m+n}'+\frac{m^3-m}{12}\delta_{m+n,0}{\bf c}'_1,\nonumber\\
&&[G_p', G_q']=2L_{p+q}'+\frac13\left(p^2-\frac14\right)\delta_{p+q, 0}{\bf c}_1',\nonumber\\
&&[L_m', G_p']=\left(\frac m2-p\right)G_{m+p}', \ [L_m',{\bf c}'_1]=0, \ \forall m, n\in\mathbb Z, \ p, q\in\mathbb Z+\frac12,
\end{eqnarray}
where  ${\bf c}'_1={\bf c}_1-(\frac32-\frac{12z^2}{\ell})\text{id}_M$. So the algebra
$${\frak s}'=\bigoplus_{n\in\mathbb Z}\left(\mathbb C L_n'+\mathbb C G_{n+\frac12}'\right)\oplus \mathbb C{{\bf c}}'_1$$
 is isomorphic to the super Virasoro algebra ${\frak s}$.  By \eqref{LHG}, we have
 \begin{equation}\label{d'bracket}
 [L'_n, H_{k}]=[G_{n+\frac12}', H_k]=0\nonumber\end{equation} and
  \begin{equation}\label{d'bracket}
 [L'_n, F_{k+\frac12}]=[G_{n+\frac12}', F_{k+\frac12}]=0, \forall n, k\in\mathbb Z\nonumber\end{equation}
  and hence  $[{\frak s}',{\mathfrak {hc}}+\mathbb C {\bf c}_2]=0$. 
  So the algebra
 \begin{equation}\frak g'={\frak s}'\oplus  ({\mathfrak {hc}}+\mathbb C {\bf c}_2)\label{direct-sum}\end{equation} is a direct sum of two ideals, and $M=U(\frak g)v=U(\frak g')v$ for any $ v\in M\setminus\{0\}$. For any $n\in\mathbb Z$, let
$$Y_n=\bigcap_{p\ge n}{\rm Ann}_{M_0}({\rm span}_{\mathbb C}\{L_p', G_{p-\frac12\tau(p)}'\}),       r_M=\min\{n\in\mathbb Z:Y_n\ne0\}, K_0=Y_{r_M}, $$
where $\tau(p)=1$ if $p\ge 1$, $\tau(p)=-1$ if $p\le 0$.

Noting that $M$ is a smooth $\frak g$-module (also smooth $\frak g'$-module), we know that $r_M<+\infty$. If $Y_n\ne0$ for all $n\in\mathbb Z$, we define $r_M=-\infty$ (see Lemma \ref{lemma-r} (i) below).
Denote by $K=U({\mathfrak {hc}})K_0$.

\begin{lemm}\label{lemma-r}
Let $M$ be a simple smooth $\frak g$-module  with level $\ell \not=0$  and central charge  $(c, z)$. Then the following statements hold.
\begin{enumerate}
\item[\rm(i)] $ r_M\ge -1$ or $r_M=-\infty$.
\item[\rm(ii)] If $r_M\ge -1$, then $K_0$ is a $\frak g^{(0,-(n_M-1))}$-module  and $H_{n_M-1}$ acts injectively on $K_0$.
\item[\rm(iii)] $K$ is a $\frak g^{(0,-\infty)}$-module and $K(z)^{\frak g}$ has a $\frak g$-module structure by $(\ref{rep2})$.
\item[\rm{(iv)}] $K_0$ and $K$ are invariant under the actions of $L_n, G_{n+\frac12}$ and $L_n', G_{n+\frac12}'$ for $n\in\mathbb N$.
\item[\rm(v)] If $r_M\ge 2$, then $L'_{r_M-1}$ acts injectively on $K_0$ and $K$.
\end{enumerate}
\end{lemm}

\begin{proof} (i) If $Y_{-2}\ne 0$, then $L'_{p}K_0=0, p\ge -2$. We deduce that ${\frak s}' K_0=0$ and hence $r_M=-\infty$.
If $Y_{-2}=0$, then $r_M\ge -1$.

(ii) For any $0\ne v\in K_0$ and $x\in {\frak g^{(0,-(n_M-1))}}$, it follows from Lemma \ref{lemma-n} (ii) that $xv\in M_0$. We first show that $L'_pxv=0, p\ge r_M$. Indeed, $L_p'H_{k}v=H_{k}L_p'v=0$ and $L_p'F_{k+\frac12}v=F_{k+\frac12}L_p'v=0$ by (\ref{d'bracket}) for any $k\geq -(n_M-1)$.
Moreover, it follows from (\ref{N=1super11}) and  (\ref{vir-bracket'}) that
$$L_p'L_nv=L_nL_p'v+[L_p', L_n]v=(n-p)L_{p+n}'v=0, \forall n\in\mathbb N,$$
$$L_p'G_{n+\frac12}v=G_{n+\frac12}L_p'v+[L_p', G_{n-\frac12}]v=(\frac{p}2-n-\frac12)G_{p+n+\frac12}'v=0, \forall n\in\mathbb N.$$
Hence, $L'_pxv=0, p\ge r_M$. Similarly  $G'_{p-\frac12\tau(r_M)}xv=0, p\ge r_M$.
That is, $xv\in K_0$, as desired.

Since $0\ne K_0  \subseteq M_0$, we see that  $H_{n_M-1}$ acts injectively on $K_0$ by Lemma \ref{lemma-n} (i).

(iii) follows from (ii).

(iv) The statement that  $K_0$ is invariant under the actions of $L_n, G_{n+\frac12}$ for $n\in\mathbb N$ follows from the definition of $K_0$ and the computations:
$$\aligned L_p'L_nK_0= G_{p-\frac12\tau(p)}'L_nK_0=0, \,\,\,L_p'G_{n+\frac12}K_0=G_{p-\frac12\tau(p)}'G_{n+\frac12}K_0=0, \forall n\ge 0, p\ge r_M, \\
 H_{k+n_M}L_nK_0= F_{k+n_M\pm1/2}L_nK_0=0, \,\,\,H_{k+n_M}G_{n+\frac12}K_0=F_{k+n_M\pm1/2}G_{n+\frac12}K_0=0, \forall n\ge 0, p\ge r_M. 
\endaligned$$
Similarly,  $K_0$ is invariant under the actions of   $L_n', G_{n+\frac12}'$ for $n\in\mathbb N$.  

Using these just established results on $K_0$, the definition of  $K$ and $K_0$, and the fact that $[\frak{s}', \frak{hc}]=0$, we can   verify that     $L'_nK\subset K, G_{n+\frac12}'K\subset K$. The statement that $K$ is invariant under the actions of $L_n, G_{n+\frac12}$ follows from
the fact that $[\frak{s}, \frak{hc}]=\frak{hc}$.

(v) If there exists $v\in K_0$ such that $L'_{r_M-1}v=0$,  then $u=G'_{r_M-\frac32}v\ne0$ by the definition of $r_M$.  However $G'_{r_M-\frac32}u=2L'_{2r_M-3}v=0$ since $r_M\ge 2$. We see that $L'_{r_M-1}u=\frac12G'_{r_M+\frac12}G'_{r_M-\frac32}u= 0$ which  contradicts the definition of $r_M$ (note that $u\in M_0$). Then $L'_{r_M-1}$ acts injectively on $K_0$, and then on $K$ since $[\frak s', \frak{hc}]=0$.
\end{proof}

\begin{prop}\label{prop-r-infty}
Let $M$ be a
simple smooth $\frak g$-module with   level $\ell \not=0$ and  central charge  $(c, z)$. If $r_M=-\infty$, then $M= K(z)^{\frak g}$. Hence $c=\frac32-\frac{12z^2}{\ell}$
 and $K$ is a simple ${\mathfrak {hc}}$-module.
\end{prop}
\begin{proof}
Since $r_M=-\infty$, we see that ${\frak s}' K_0=0$. This together with (\ref{vir-bracket'}) implies that $c_1=\frac32-\frac{12z^2}{\ell}$.  Noting that $[{\frak s}',  {\mathfrak {hc}}+\mathbb C{ {\bf c}}_2]=0$, we further obtain that ${\frak s}' K=0$, that is,  $L_nv=\bar L_nv, G_{n+\frac12}v=\bar G_{n+\frac12}v\in K$ for any $v\in K$ and $n\in\mathbb Z$. Hence $K(z)^{\frak g}$ is a $\frak g$-submodule of $M$, yielding that $M= K(z)^{\frak g}$. In particular, $K$ is a simple ${\mathfrak {hc}}$-module.
\end{proof}

\begin{prop}\label{prop-r2-n2}
Let $M$ be a simple smooth $\frak g$-module  with level $\ell \not=0$  and central charge  $(c, z)$. If $r_M\ge 2$ and $n_M\ge 2$, then
$K_0$ is a simple $\frak g^{(0,-(n_M-1))}$-module and
$M\cong {\rm Ind}_{\frak g^{(0,-(n_M-1))}}^{\frak g}K_0$.
\end{prop}

\begin{proof}We first show that  ${\rm Ind}_{\frak g^{(0,-(n_M-1))}}^{\frak g^{(0,-\infty)}}K_0 \cong K$ as  $\frak g^{(0,-\infty)}$ modules. For that, let
\begin{eqnarray*}
\phi:\, {\rm Ind}_{\frak g^{(0,-(n_M-1))}}^{\frak g^{(0,-\infty)}}K_0  &\longrightarrow &K\\
\sum_{\mj\in\mathbb{M}, \ml\in\mathbb M_1} H^{\mj}F^{\ml}\otimes v_{\mj, \ml}&\mapsto &
\sum_{\mj\in\mathbb{M}, \ml\in\mathbb M_1} H^{\mj}F^{\ml} v_{\mj, \ml},
\end{eqnarray*} where $ v_{\mj, \ml}\in K_0\setminus\{0\}$, $H^{\mj}F^{\ml}=\cdots H^{j_2}_{-2-(n_M-1)}H_{-1-(n_M-1)}^{j_1}\cdots F^{l_2}_{-\frac32-(n_M-1)}F^{l_1}_{-\frac12-(n_M-1)}\in U({\mathfrak {hc}})$. Then $\phi$ is a  $\frak g^{(0,-\infty)}$-module epimorphism and $\phi|_{K_0}$ is one-to-one.

{\bf Claim}. Any nonzero submodule $V$ of ${\rm Ind}_{\frak g^{(0,-(n_M-1))}}^{\frak g^{(0,-\infty)}}K_0$ does not  intersect with $K_0$ trivially.

Assume $V\cap K_0=0$. Let $v=\sum_{\mj\in\mathbb{M}, \ml\in\mathbb M_1} H^{\mj}F^{\ml}\otimes v_{\mj, \ml}\in V\backslash K_0 $ with  minimal deg$(v):=(\bmj, \bml)$. Then $(\bf 0, \bf0)\prec (\bmj, \bml)$.

Let $p=\text{min}\{s:i_s\ne 0\}$. Since  $H_{p+n_M-1}v_{\mj, \ml}=0$, we have $H_{p+n_M-1}H^{\mj}F^{\ml}\otimes v_{\mj, \ml}=[H_{p+n_M-1},H^{\mj}]F^{\ml}v_{\mj, \ml}$.
If  $j_p=0$ then $H_{p+n_M-1}H^{\mj}F^{\ml}\otimes v_{\mj, \ml}=0$, and if $j_p\ne 0$, noticing
the level $\ell\ne 0$,
then $[H_{p+n_M-1},h^{\mj}F^{\ml}]=\lambda H^{\mj-\epsilon_p}F^{\ml}$ for
some  $\lambda\in\mathbb C^*$ and hence $$\text{deg}([H_{p+n_M-1},H^{\mj}]F^{\ml}v_{\mj, \ml})=(\mj-\epsilon_p, \ml)\preceq (\bmj-\epsilon_p, \bml),$$
where the equality holds if and only if $\bmj=\mj$. Hence $\deg(H_{p+n_M-1}v)=(\bmj-\epsilon_p, \bml)\prec (\bmj, \bml)$ and $H_{p+n_M-1}v\in V$,   contrary to the choice of $v$. 

Similarly,  set $p=\text{min}\{s:j_s\ne 0\}$, we can also get a contradiction.  So the claim holds.

By the claim, as $\frak g$-modules, we have
$${\rm Ind}_{\frak g^{(0,-(n_M-1))}}^{\frak g}K_0\cong{\rm Ind}_{\frak g^{(0,-\infty)}}^{\frak g}({\rm Ind}_{\frak g^{(0,-(n_M-1))}}^{\frak g^{(0,-\infty)}}K_0)\cong{\rm Ind}_{\frak g^{(0,-\infty)}}^{\frak g}K.$$

By the way, it is clear that ${\rm Ind}_{\frak g^{(0,-\infty)}}^{\frak g}K\cong {\rm Ind}_{{\frak s}'^{(0)}}^{{\frak s}'}K$ as vector spaces.
Moreover, we have the following $\frak g$-module epimorphism
\begin{eqnarray*}
\pi: {\rm Ind}_{\frak g^{(0,-\infty)}}^{\frak g}K={\rm Ind}_{{\frak s}'^{(0)}}^{{\frak s}'}K&\rightarrow& M,\cr
\sum_{\mi\in\mathbb{M},\mk\in\mathbb M_1}L'^{\mi}G'^{\mk}\otimes v_{\mi, \mk}&\mapsto& \sum_{\ml\in\mathbb{M}, \mk\in\mathbb M_1}L'^{\mi}G'^{\mk}v_{\mi, \mk},
\end{eqnarray*}
where $L'^{\mi}=\cdots (L'_{-2})^{i_2}(L'_{-1})^{i_1}$ and $G'^{\mk}=\cdots(G'_{-\frac32})^{k_2}(G_{-\frac12}')^{k_1}$.
We see that $\pi$ is also an ${\frak s}'$-module epimorphism. 

By Lemma \ref{lemma-r} (v), we see that $L_{r_M-1}'$ act injectively on  $K$.
By the proof of Theorem 3.1 in \cite{LPX1} we know that any nonzero ${\frak s}'$-submodule of ${\rm Ind}_{{\frak s}'^{(0)}}^{{\frak s}'}K$ contain nonzero vectors of $K$. Note that $\pi|_K$ is one-to-one, we see  that the image of any nonzero $\frak g$-submodule (and hence ${\frak s}'$-submodule) of ${\rm Ind}_{\frak g^{(0,-\infty)}}^{\frak g}K$  must be a nonzero $\frak g$-submodule of $M$ and hence be the  whole module $M$, which forces that the kernel of $\pi$ must be  $ 0$. Therefore, $\pi$ is an isomorphism. Since $M$ is simple, we see  $K_0$ is a simple $\frak g^{(0,-(n_M-1))}$-module.\end{proof}

\begin{lemm}\label{lemma-eigenvalue} Let $M$ be a simple smooth $\frak g$-module  with level $\ell \not=0$.   If  $r_M=1$, then $L_0'$ has an  eigenvector in $K$.
\end{lemm}
\begin{proof}Lemma \ref{lemma-r} (iv) means that $K$ is a $\frak g'^{(0,-\infty)}$-module.
Assume that any finitely generated $\mathbb C[L'_0]$-submodule of $K$ is a free $\mathbb C[L'_0]$-module. By Lemma \ref{free-L0} we see that
the following $\frak g'$-module homomorphism
\begin{eqnarray*}
\varphi:{\rm Ind}_{\frak g'^{(0,-\infty)}}^{{\frak g}'}K={\rm Ind}_{{\frak s}'^{(0)}}^{{\frak s}'}K&\longrightarrow & M, \cr
 x\otimes u&\mapsto &xu, x\in U({\frak s}'), u\in K.
\end{eqnarray*}
is an isomorphism. So
$M={\rm Ind}_{{\frak s}'^{(0)}}^{{\frak s}'}K$, and consequently, $K$ is a simple
$\frak g'^{(0,-\infty)}$-module. Since $r_M=1$ and ${\frak s}'^+K=0$,  $K$ can be seen as a simple module over the Lie superalgebra ${\mathfrak {hc}}\oplus \mathbb C{\bf c}_2\oplus\mathbb C L'_0$, where $\mathbb C L'_0$ lies in the center of the Lie superalgebra. Schur's lemma tells us that $L'_0$ acts as a scalar  on $K$,  a contradiction. So this case will not occur.

Therefore,   there exists some  finitely generated $\mathbb C[L'_0]$-submodule $W$ of $K$ that is not a free $\mathbb C[L'_0]$-module. Since $\mathbb C[L'_0]$ is a principal ideal domain,  by the structure theorem of finitely generated modules over a principal ideal domain,   there exists a monic polynomial $f(L'_0)\in\mathbb C[L'_0]$  with minimal positive degree  and nonzero element $u\in W$ such that $f(L'_0)u=0$. Write $f(L'_0)=\Pi_{1\le i\le s}(L_0'-\lambda_i)$, $\lambda_1,\cdots,\lambda_s\in\mathbb C$. Denote $w:=\prod_{i=1}^{s-1}(L_0'-\lambda_{i})u\neq 0$, we see $(L_0'-\lambda_s)w=0$  where we make convention that $w=u$ if $s=1$. Then $w$ is a desired eigenvector of $L_0'$.
\end{proof}

\begin{prop} \label{prop-tensor} Let $M$ be a simple smooth $\frak g$-module  with level $\ell \not=0$  and central charge  $(c, z)$.   If $r_M=0,\pm 1$, then $K$ is
 a simple ${\mathfrak {hc}}$-module  and
$M \cong  U^{\frak g}\otimes K(z)^{\frak g}$ for some simple module $U\in \mathcal{R}_{{\frak s}}$.
\end{prop}

\begin{proof} If $r_M=1$, then by Lemma \ref{lemma-eigenvalue} we know that there exists  $0\ne u\in K$ such that $L_0'u=\lambda u$ for some $\lambda \ne 0$; if $r_M=0,-1$, then $L_0'K=0$. In summary, for all  the three cases,  $L_0'$ has an eigenvector in $ K$.
Since $M$ is a simple $\frak g'$-module,  Schur's lemma implies that ${\bf c}'_1, {\bf c}_2, {\bf c}_3$ act as scalars on $M$. So $M$ is a weight $\frak g'$-module, and $K$ is a  weight module for $\frak g'^{(s_M, \, -\infty)}$, where $s_M=r_M-\delta_{r_M, 1}$. Take a weight vector $u_0\in K$ with $L'_0u_0=\lambda_0u_0$ for some $\lambda_0\in\mathbb C$.

Set $K'=U({\mathfrak {hc}})u_0$, which is an ${\mathfrak {hc}}$-submodule of $K$. Now we define  the $\frak g'$-module $K'^{{\frak g}'}$ with trivial action of  ${\frak s}'$. 
Let $\mathbb C v_0$ be the one-dimensional $\frak g'^{(s_M, \, -\infty)}$-module defined by 
\begin{eqnarray*}&&L'_0 v_0=\lambda_0v_0, \  {\bf c}_1'v_0=(c-\frac32+\frac{12z^2}{\ell})v_0,  \\
&&L_n'v_0=G_{m+\frac12}'v_0=H_{k}v_0=F_{k+\frac12}v_0={\bf c}_2v_0={\bf c}_3v_0=0, \,\,\, 0\ne n\ge r_M, m\ge s_M, k\in\mathbb Z.
\end{eqnarray*}
Then $\mathbb C v_0\otimes K'^{{\frak g}'}$ is a $\frak g'^{(s_M, \, -\infty)}$-module with central charge $(c_1-\frac32+\frac{12z^2}{\ell}, z)$ and level $\ell$.  There is a $\frak g'^{(s_M, \, -\infty)}$-module homomorphism
\begin{eqnarray*}
\varphi_{K'}: \mathbb C v_0\otimes K'^{{\frak g}'}& \longrightarrow & M,\cr
 v_0\otimes u &\mapsto & u, \forall u\in K',
\end{eqnarray*}
which is injective and can be extended to be the following $\frak g'$-module epimomorphism
\begin{eqnarray*}
\varphi:{\rm Ind}_{{\frak g}'^{(s_M, \, -\infty)}}^{\frak g'}(\mathbb C v_0\otimes K'^{{\frak g}'})&\longrightarrow & M, \cr
 x(v_0\otimes u)&\mapsto &xu, x\in U(\frak g'), u\in K'.
\end{eqnarray*}
By Lemma \ref{induced} we know that
$${\rm Ind}_{\frak g'^{(s_M, \, -\infty)}}^{{\frak g}'}(\mathbb C v_0\otimes K'^{{\frak g}'})\cong ({\rm Ind}_{{\frak g'}^{(s_M, \, -\infty)}}^{{\frak g}'}\mathbb C v_0)\otimes K'^{{\frak g}'}
=({\rm Ind}_{{\frak s}'^{(s_M)}}^{{\frak s}'}\mathbb C v_0)^{{\frak g}'}\otimes K'^{{\frak g}'}.
$$Then we have the following $\frak g'$-module epimorphism
\begin{eqnarray*}
\varphi':({\rm Ind}_{{\frak s}'^{(s_M)}}^{{\frak s}'}\mathbb C v_0)^{{\frak g}'}\otimes K'^{{\frak g}'}&\longrightarrow& M,\cr
  xv_0\otimes u&\mapsto& xu, x\in U({\frak s}'), u\in K'.
\end{eqnarray*}
Note that $({\rm Ind}_{{\frak s}'^{(s_M)}}^{{\frak s}'}\mathbb C v_0)^{{\frak g}'}\otimes K'^{{\frak g}'}\cong{\rm Ind}_{{\frak s}'^{(s_M)}}^{{\frak s}'}(\mathbb C v_0\otimes K'^{{\frak g}'})$ as ${\frak s}'$-modules, and $\varphi'$ is also an ${\frak s}'$-module epimorphism, $\varphi'|_{\mathbb C v_0\otimes K'^{{\frak g}'}}$ is one-to-one, and $({\rm Ind}_{{\frak s}'^{(s_M)}}^{{\frak s}'}\mathbb C v_0)^{{\frak g}'}\otimes K'^{{\frak g}'}$ is a highest weight ${\frak s}'$-module.

Let $V={\rm Ind}_{{\frak s}'^{(s_M)}}^{{\frak s}'}\mathbb C v_0$ and  $\mathfrak{K}=\text{Ker}(\varphi')$. Since
$\mathbb C v_0\otimes K'^{{\frak g}'}=\{u\in V^{{\frak g}'}\otimes K'^{{\frak g}'}\mid L_0^{\prime}u=\lambda_0u\}$, 
we have \begin{equation}(\mathbb C v_0\otimes K'^{{\frak g}'})\cap \mathfrak{K}=0.\label{maximal}\end{equation}
Let $\mathfrak{K}^{\prime}$  be the sum of all ${\frak s}'$-submodules $W$ of $V^{{\frak g}'}\otimes K'^{{\frak g}'}$
with $W\cap(\mathbb C v_0\otimes K'^{{\frak g}'})=0$, that is, the unique maximal (weight) ${\frak s}'$-submodule of $V^{{\frak g}'}\otimes K'^{{\frak g}'}$ with trivial intersection with  $(\mathbb C v_0\otimes K'^{{\frak g}'})$. It is obvious that $\mathfrak{K}\subseteq\mathfrak{K}'$ by \eqref{maximal}.

Next we further show that $\mathfrak{K}=\mathfrak{K}'$. For that, take any ${\frak s}'$-
submodule $W$ of $V^{{\frak g}'}\otimes K'^{{\frak g}'}$ such that $W\cap(\mathbb C v_0\otimes K'^{{\frak g}'})=0$. Then for any weight vector $w=\sum_{\ml\in \mathbb{M}}
L'^{\ml}G'^{\mk}v_0\otimes u_{\ml, \mk}\in W$, where $u_{\ml, \mk}\in {K'}^{{\frak g}'},$
$$ L'^{\ml}=\cdots (L'_{-2})^{l_2}(L'_{-1})^{l_1}\text{ and }G'^{\mk}=\cdots (G'_{-\frac32})^{k_2}(G'_{-\frac12})^{k_1}\text{ if }r_M=1,0,\text{ or }$$
$$L'^{\ml}=\cdots (L'_{-2})^{l_2}\text{ and }G'^{\mk}=\cdots (G'_{-\frac32})^{k_2}\text{ if }r_M=-1,$$ and all ${\rm{w}}(\ml, \mk)\ge \frac12$  are equal. Note that
$H_{k}w=\sum_{\ml\in \mathbb{M}, \mk\in\mathbb M_1}L'^{\ml}G'^{\mk}v_0\otimes H_{k}u_{\ml, \mk}$  either
equals to $0$ or has the same weight as $w$ under the action of $L_0^{\prime}$.
So $U(\frak g')W\cap  (\mathbb C v_0\otimes K'^{{\frak g}'})=0$, i.e., $U(\frak g')W\subset \mathfrak{K}'$. 
Moreover, we have $U(\frak g')\mathfrak{K}'\cap  (\mathbb C v_0\otimes K'^{{\frak g}'})=0$. The maximality of $\mathfrak{K}'$ forces that $\mathfrak{K}'=U(\frak g')\mathfrak{K}'$ is a  proper $\frak g'$-submodule of $V^{{\frak g}'}\otimes K'^{{\frak g}'}$. Since $\mathfrak{K}$ is a maximal proper $\frak g'$-submodule of $V^{{\frak g}'}\otimes K'^{{\frak g}'}$, it follows that $\mathfrak{K}'\subseteq\mathfrak{K}$. Thus 
$\mathfrak{K}=\mathfrak{K}'$. 

Now $\frak K$ is just the unique maximal (weight) ${\frak s}'$-submodule of $V^{{\frak g}'}\otimes K'^{{\frak g}'}$ with trivial intersection with  $(\mathbb C v_0\otimes K'^{{\frak g}'})$.
By Lemma \ref{singular-vector} we know that $\mathfrak{K}$ is generated by
$\mathbb C P_1 v_0\otimes K'^{{\frak g}'}$ and $\mathbb C P_2 v_0\otimes K'^{{\frak g}'}$. Let  $V'$ be the  maximal  $\mathfrak s'$-submodule of $V$ generated by $P_1v_0$ and $P_2v_0$, then $\mathfrak{K}=V'^{{\frak g}'}\otimes K'^{{\frak g}'}$. Therefore,
\begin{equation}\label{cong}
M\cong (V^{{\frak g}'}\otimes K'^{{\frak g}'})/(V'^{{\frak g}'}\otimes K'^{{\frak g}'})\cong (V/V')^{{\frak g}'}\otimes K'^{{\frak g}'},\end{equation}
which forces  that $K'^{{\frak g}'}$ is a simple $\frak g'$-module and hence a simple ${\mathfrak {hc}}$-module. So   $K'$ is  a simple ${\mathfrak {hc}}$-module. By Corollary \ref{tensor} we know there exists a simple ${\frak s}$-module $U\in {\mathcal{R}}_{{\frak s}}$ such that $M\cong U^{\frak g}\otimes K'(z)^{\frak g}$. From this isomorphism and some computations we see that $K_0\subseteq v_0\otimes K'(z)^{\frak g}$, where $v_0$ is a highest weight vector. So $K=K'$.\end{proof}

We are now in a position to present the following main result on classifications of simple smooth $\frak g$-modules with nonzero level.

\begin{theo}\label{mainthm}
Let $M$ be a simple smooth $\frak g$-module  with level $\ell \not=0$  and central charge  $(c, z)$. The invariants $ n_M, r_M$ of $M$, $ K_0, K$ are defined as before. Then
\begin{equation*}
M\cong\begin{cases} K(z)^{\frak g}, &{\text{ if }}r_M=-\infty,\cr
U^{\frak g}\otimes H(z)^{\frak g}, &{\text{ if }}\ r_M=0, \pm1 \text{ or } n_M=0, 1, \cr
{\rm Ind}_{\frak g^{(0,-(n_M-1))}}^{\frak g}K_0, &{\text { otherwise,
}}\end{cases}
\end{equation*}
for  some simple $U\in \mathcal{R}_{{\frak s}}, H\in \mathcal R_{\mathfrak {hc}}$.
\end{theo}

\begin{proof}
The assertion follows directly from Proposition \ref{prop-n=01}, Proposition \ref{prop-r-infty}, Proposition \ref{prop-r2-n2} and Proposition \ref{prop-tensor}.
\end{proof}

\section{Simple smooth ${\frak g}$-modules at level zero}

In this section we will determine   simple  smooth ${\frak g}$-module $M$ of level zero with a mild condition.  

\subsection{Constructing simple smooth ${\frak g}$-modules}

For any $q\in\mathbb N, c, d, z\in\mathbb C$, let $V$ be 
a simple $\mathfrak g^{(0, -q)}$-module with the actions of $H_0, {\bf c}_1, {\bf c}_2, {\bf c}_3$ as scalars $d, c, z, l$, respectively.
In this section we mainly consider the corresponding induced module $\mathrm{Ind}_{\frak g^{(0, -q)}}^{\frak g}V$.
It is well known that  $\mathrm{Ind}_{\frak g^{(0, -q)}}^{\frak g}V$ becomes the Verma module if $q=0$ and $V$ is a one dimensional  $\frak g^{(0, 0)}$-module;
$\mathrm{Ind}_{\frak g^{(0, -q)}}^{\frak g}V$ becomes a Whittaker module if $q=0$ and $V={\rm Ind}_{\frak g^{(k, 0)}}^{\frak g^{(0, 0)}}W$, where $W$ is a finite dimensional simple $\frak g^{(k, 0)}$-module for some $k\in\mathbb Z_+$ (see Section 6).


\begin{theo}\label{simple-theo} Let  $c, z, d\in\mathbb C, q\in\mathbb N$, and  $V$   be 
a simple $\mathfrak g^{(0, -q)}$-module with central charge $(c, z)$, level 0,  and the action of $H_0$ as a scalar $d$. Assume that there exists  $t\in\mathbb Z_+$ satisfying the following two conditions:
\begin{itemize}
\item[{\rm (a)}] the action of $H_{t}$ on $V$ is injective;
\item[{\rm (b)}] $H_iV=0$ for all $i>t$ and $L_jV=0$ for all $j>t+q$.
\end{itemize}
Then
\begin{itemize}
\item[{\rm (i)}] $G_{j-\frac12}V=0$ for all $j>t+q$, and $F_{i-\frac12}V=0$ for all $i>t$,
\item[\rm (ii)]   the induced module $\mathrm{Ind}_{\frak g^{(0, -q)}}^{\frak g}V$ is a simple ${\frak g}$-module.
\end{itemize}
\end{theo}

\begin{proof}
(i) By (b), for $j\geq t+q$ we have $G_{j+\frac12}^2V=L_{2j+1}V=0$. If $G_{j+\frac12}V=0$ we are done. Otherwise,   $W=G_{j+\frac12}V$    is a proper subspace of $V$.   For $r\in\mathbb{Z}_+$, we have
$$F_{i+\frac12}W\subseteq W,\,\,\,H_iW\subseteq W,\forall i\ge -q,$$
$$
G_{r-\frac{1}{2}}W=G_{r-\frac{1}{2}}G_{j+\frac12}V=L_{r+j}V-G_{j+\frac12}G_{r-\frac{1}{2}}V\subset G_{j+\frac12}V=W,
$$
$$
2L_rW=[G_{r-\frac{1}{2}},G_{\frac{1}{2}}]G_{j+\frac12}V
=G_{r-\frac{1}{2}}G_{\frac{1}{2}}G_{j+\frac12}V+G_{\frac{1}{2}}G_{r-\frac{1}{2}}G_{j+\frac12}V\subset W.
$$
It follows that $W$ is a proper $\mathfrak g^{(0, -q)}$-submodule of $V$. Then $W=G_{j+\frac12}V=0$ for $j\geq t+q$ since $V$ is simple.

For any $i>t$, $iF_{i+\frac12}V=[H_i, G_{\frac12}]V=0$. Note that $F_{t+\frac12}^2V=0$. If $F_{t+\frac12}V=0$ we are done. Otherwise   $U=F_{t+\frac12}V$    is a proper subspace of $V$.  We can easily verify that 
$$F_{i+\frac12}U\subseteq U,\,\,\,H_iU\subseteq U,\forall i\ge -q,$$
$$
G_{r+\frac{1}{2}}U\subseteq U,
\,\,\,
L_rU\subset U, \forall r\ge 0.
$$
So $U$ is a $\mathfrak g^{(0, -q)}$-submodule of $V$.
Thus $U=F_{t+\frac12}V=0$.

 (ii) Suppose that $W$ is a nonzero proper $\mathfrak g$-submodule of $\mathrm{Ind}_{\frak g^{(0, -q)}}^{\frak g}V$. Now we are going to deduce some contradictions. Let $v\in W\setminus \{0\}$ such that  $\mathrm{deg}(v)=(\bmi,\bmj,\bmk, \bml)$ is minimal. Write
\begin{equation}
v=\sum_{(\mi, \mj, \mk, \ml)\in \supp(v)} L^{\mi}H^{\mj}G^{\mk}F^{\ml}v_{\mi, \mj, \mk, \ml},\label{v-express1}
\end{equation}
where all  $v_{\mi, \mj, \mk, \ml}\in V$ and only finitely many of them are nonzero. Note that $\mj=\mathbf{0}$ or $\hat{j}>q$, and 
$\ml=\mathbf{0}$ or $\hat{l}>q$ for any  $(\mi, \mj, \mk, \ml)\in \mathrm{supp}(v)$.


{\bf Claim 1.} $\bmk=\bf 0$.

If $\bmk\neq\mathbf{0}$, then $\hat{\underline{k}}>0$.
 We will show that $\deg( F_{\hat{\underline k}+t-\frac12}v)=(\bmi, \bmj, \bmk^{\prime}, \bml).$
It suffices to consider those $v_{\mi, \mj, \mk, \ml}$
with $$L^{\mi}H^{\mj}G^{\mk}F^{\ml}v_{\mi, \mj, \mk, \ml}\neq0.$$

 For any  $(\mi, \mj, \mk, \ml)\in \mathrm{supp}(v)$, since  $F_{\hat{\underline  k}+t-\frac12}v_{\mi, \mj, \mk, \ml}=0$
we see that
\begin{eqnarray*}&&F_{\hat{\underline k}+t-\frac12}L^{\mi}H^{\mj}G^{\mk}F^{\ml}v_{\mi, \mj, \mk, \ml}=[F_{\hat{\underline k}+t-\frac12}, L^{\mi}]H^{\mj}G^{\mathbf{k}}F^{\ml}v_{\mi, \mj, \mk, \ml}\pm L^{\mi}H^{\mj}[F_{\hat{\underline k}+t-\frac12}, G^{\mathbf{k}}]F^{\ml}v_{\mi, \mj, \mk, \ml}.
\end{eqnarray*}
Since $\text{w}([F_{\hat{\underline k}+t-\frac12}, L^{\mi}]H^{\mj}G^{\mathbf{k}}F^{\ml}v_{\mi, \mj, \mk, \ml})\le \text{w}(\mi, \mj, \mk, \ml)-\hat{\underline k}$ (we have used the property of $(\bmi,\bmj,\bmk, \bml)$) and  $\text{w}([L^{\mi}H^{\mj}[F_{\hat{\underline k}+t-\frac12}, G^{\mathbf{k}}]F^{\ml}v_{\mi, \mj, \mk, \ml})\le \text{w}(\mi, \mj, \mk, \ml)-\hat{\underline k}+\frac12$ by (a), (b) and (i), we deduce that 
\begin{equation}\text w( F_{\hat{\underline k}+t-\frac12}L^\mi H^{\mj}G^{\mk}F^{\ml}v_{\mi, \mj, \mk, \ml})\le\text w(\mi, \mj, \mk, \ml)-\hat{\underline k}+\frac12.\label{degree-F}\end{equation}

 For any  $(\mi, \mj, \mk, \ml)\in \mathrm{supp}(v)$ with
$
{\rm w}(\mi, \mj, \mk, \ml)<{\rm w}(\bmi, \bmj, \bmk, \bml),
$
by \eqref{degree-F} we see that $$\text w( F_{\hat{\underline k}+t-\frac12}L^\mi H^{\mj}G^{\mk}F^{\ml}v_{\mi, \mj, \mk, \ml})\le\text w(\mi, \mj, \mk, \ml)-\hat{\underline k}+\frac12<{\rm w}(\bmi, \bmj, \bmk, \bml)-\hat{\underline k}+\frac12.$$So
$$
\deg F_{\hat{\underline k}+t-\frac12}L^{\mi}H^{\mj}G^{\mk}F^{\ml}v_{\mi, \mj, \mk, \ml}\prec (\bmi, \bmj, \bmk^{\prime}, \bml).
$$

Now we suppose that ${\rm w}(\mi, \mj, \mk, \ml)={\rm w}(\bmi, \bmj, \bmk, \bml)$
 and $ \mathbf{k} \prec \bmk$ and denote  $$\mathrm{deg}(F_{\hat{\underline k}+t-\frac12}L^{\mi}H^{\mj}G^{\mk}F^{\ml}v_{\mi, \mj, \mk, \ml})
=(\mi_1, \mj_1, \mk_1, \ml_1)\in \mathbb{M}\times \mathbb{M}_1.$$

 { If $\hat{k}>\hat{\underline k}$, then ${\rm w}(\mi_1, \mj_1, \mk_1, \ml_1)<{\rm w}(\bmi, \bmj, \bmk^\prime, \bml)$} and  
$$\mathrm{deg}(F_{\hat{{\underline k}}+t-\frac12}L^{\mi}H^{\mj}G^{\mk}F^{\ml}v_{\mi, \mj, \mk, \ml})
=(\mi_1, \mj_1, \mk_1, \ml_1)\prec(\bmi, \bmj, \bmk^\prime, \bml).$$
If $\hat{k}=\hat{\underline k}$,
then ${\rm w}(\mi_1, \mj_1, \mk_1, \ml_1)={\rm w}(\mi, \mj, \mk^\prime, \ml)$.
Since $ \mathbf{k}^{\prime} \prec\bmk^{\prime}$, we also have $(\mi_1, \mj_1, \mk_1, \ml_1)\prec(\bmi, \bmj, \bmk^\prime, \bml)$.

If ${\rm w}(\mi, \mj, \mk, \ml)={\rm w}(\bmi, \bmj, \bmk, \bml)$ and $\mathbf{k}=\bmk$, it is easy to see that
$$\mathrm{deg}(L^{\mi}H^{\mj}[F_{\hat{\underline k}+t-\frac12}, G^{\mathbf{k}}]F^{\mathbf{l}}v_{\mi, \mj, \mk, \ml})=
(\mi, \mj, \mk^\prime, \ml)\preceq (\bmi, \bmj, \bmk^\prime, \bml),$$
$$\mathrm{deg}([F_{\hat{\underline k}+t-\frac12}, L^{\mi}]H^{\mj}G^{\mk}F^{\ml}v_{\mi, \mj, \mk, \ml})\prec(\mi, \mj, \mk^\prime, \ml)$$
where the equality holds if and only if $(\mi, \mj, \mk, \ml)=(\bmi, \bmj, \bmk, \bml)$.

So  $F_{\hat{\underline k}+t-\frac12}v\in W$ and $\deg( F_{\hat{\underline k}+t-\frac12}v)=(\bmi, \bmj, \bmk^{\prime}, \bml),$
which contradicts the choice of $v$. Consequently,  Claim 1 holds. 

{\bf Claim 2.} $\bmi=\bf 0$.

If  $\bmi\neq\mathbf{0}$, then $\hat{\underline{i}}>0$. We will show that $\deg( H_{\hat{\underline i}+t}v)=(\bmi^\prime, \bmj, \bf0, \bml).$
It suffices to consider those $v_{\mi, \mj, \mk, \ml}$
with $$L^{\mi}H^{\mj}G^{\mk}F^{\ml}v_{\mi, \mj, \mk, \ml}\neq0.$$

 For any  $(\mi, \mj, \mk, \ml)\in \mathrm{supp}(v)$ with
$
{\rm w}(\mi, \mj, \mk, \ml)<{\rm w}(\bmi, \bmj, \bf0, \bml),
$
\begin{eqnarray*}&&H_{\hat{\underline i}+t}L^{\mi}H^{\mj}G^\mk F^{\ml}v_{\mi, \mj, \mk, \ml}=[H_{\hat{\underline i}+t}, L^{\mi}]H^{\mj}G^\mk F^{\ml}v_{\mi, \mj, \mk, \ml}\pm L^{\mi}H^\mj[H_{\hat{\underline i}+t}, G^{\mk}]F^{\ml}v_{\mi, \mj, \mk, \ml}.
\end{eqnarray*}
From (a), (b), and (i), we see that $${\rm w}( L^{\mi}H^\mj[H_{\hat{\underline i}+t}, G^{\mk}]F^{\ml}v_{\mi, \mj, \mk, \ml})\le{\rm w}( {\mi, \mj, \mk, \ml})-\hat {\underline i}-\frac12\ {\rm and}\ {\rm w}([H_{\hat{\underline i}+t}, L^{\mi}]H^{\mj}G^\mk F^{\ml})\le {\rm w}( {\mi, \mj, \mk, \ml})-\hat {\underline i}.$$ Then
$\text w(H_{\hat{\underline i}+t}L^{\mi}H^{\mj}G^\mk F^{\ml}v_{\mi, \mj, \mk, \ml}) \le {\rm w}(\mi, \mj,\mk, \ml)-\hat {\underline i}<{\rm w}(\bmi', \bmj, \mathbf0, \bml)$.
So $$\text{deg}(H_{\hat{\underline i}+t}L^{\mi}H^{\mj}G^\mk F^{\ml}v_{\mi, \mj, \mk, \ml})\prec  (\bmi', \bmj, \mathbf0, \bml).$$

Now we suppose that ${\rm w}(\mi, \mj, \bf0, \ml)={\rm w}(\bmi, \bmj, \bf0, \bml)$ and $(\mi, \mj, \bf0, \ml)\prec(\bmi, \bmj, \bf0, \bml)$. Then $\hat i\ge \hat {\underline i}$. It is clear that 
$$\mathrm{deg}([H_{\hat{\underline i}+t}, L^{\mi}]H^{\mj}F^{\ml}v_{\mi, \mj, \bf, \ml})\preceq(\bmi^\prime, \bmj, \bf0, \bml), $$
where the equality holds if and only if $(\mi, \mj, \bf0, \ml)=(\bmi, \bmj, \bf0, \bml)$.

So  $H_{\hat{\underline i}+t}v\in W$ and $\deg( H_{\hat{\underline i}+t }v)=(\bmi^{\prime}, \bmj, \bmk, \bml),$
which contradicts the choice of $v$. Consequently,  Claim 2 holds. 

From Claims 1 and 2 we know that $\mi= \mk=\bf 0$ for $(\mi,\mj, \mk, \ml)\in {\rm supp}(v)$ with $ {\rm w}(\mi,\mj, \mk, \ml)={\rm w}(\bmi , \bmj, \bmk, \bml)$.

Set $(\mi_1, \mj_1, \mk_1, \ml_1):= {\rm max}\,\{ (\mi,\mj, \mk, \ml)\in {\rm supp}(v)\mid {\rm w}(\mi,\mj, \mk, \ml)<{\rm w}({\bf0}, \bmj, {\bf 0}, \bml)\}$. Using the arguments in Claims 1 and 2, we deduce that   ${\rm deg}\,F_{\hat{k}_1+t-\frac12}v=(\mi_1, \mj_1, \mk_1^\prime, \ml_1)$ if $\mk_1\ne 0$; ${\rm deg}\,H_{\hat{i}_1+t}v=(\mi_1^\prime, \mj_1, \mk_1, \ml_1)$ if $\mk_1=\bf0$ and $\mi_1\ne 0$. 

Repeating these arguments we can take 
 \begin{equation}v=\sum H^{\mj}F^{\ml}v_{\mj, \ml}\in W\setminus\{0\}\label{v-express2}\end{equation} with ${\rm deg}_{\frak {hc}} (v)=(\bmj, \bml)$. If $\hat{\hat{j}}>0$, then ${\rm deg}\,L_{\hat{\hat j}+t}v=(\bmj^{\prime\prime}, \bml)$.

Now we can take $v=\sum F^{\ml}v_{\ml}\in W\setminus\{0\}$ with ${\rm deg} (v)=\bml$, if $\hat{\hat{l}}>0$, then ${\rm deg}\,G_{\hat{\hat j}+t-\frac12}v=\bml^{\prime\prime}$. 

 So such $W$ does not exist, which implies  the simplicity of $\mathrm{Ind}_{\frak g^{(0, -q)}}^{\frak g}V$.
\end{proof}

Now we consider the case of $t=0$ for Theorem \ref{simple-theo}.

\begin{theo}\label{simple-theo1}  
Let  $c, z, d\in\mathbb C, q\in\mathbb N$, and  $V$   be 
a simple $\mathfrak g^{(0, -q)}$-module with central charge $(c, z)$, level 0,  and the action of $H_0$ as a scalar $d$.
 Assume that  $V$ satisfies the following conditions:
\begin{itemize}
\item[{\rm (a)}] $d+(n+1)z\ne0$ for any $n\in\mathbb Z^*$; 
\item[{\rm (b)}] $L_jV=0$ for all $j>q$ and $H_iV=0$ for all $i>0$.
\end{itemize}
Then
\begin{itemize}
\item[{\rm (i)}] $G_{j-\frac12}V=0$ for all $j>q$, and $F_{i-\frac12}V=0$ for all $i>0$,
\item[{\rm (ii)}]  the induced module $\mathrm{Ind}_{\frak g^{(0, -q)}}^{\frak g}V$ is a simple ${\frak g}$-module.
\end{itemize}
\end{theo}

\begin{proof}
(i) The proof is the same as that of (i) in Theorem \ref{simple-theo}.

(ii) It is enough to show that any nonzero proper submodule of ${\rm Ind}_{\frak g^{(0, -q)}}^{\frak g}V$ has a nonzero intersection with $V$ (we point out that the proof for this statement will not use the simplicity on $V$).  To the contrary, suppose that $W$ is a nonzero proper $\mathfrak g$-submodule of ${\rm Ind}_{\frak g^{(0,-q)}}^{\frak g}V$ with $W\cap V=0$. Now we are going to deduce some contradictions. 

Choose  $v\in W$ such that deg$_{\frak s}(v)$ minimal in all nonzero elements of $W$. Then
 \begin{equation}v=\sum_{\stackrel{ (\mi, \mk)\in {\rm supp}_{\frak s}(v)}{\text w(\mi, \mk)=|v|_{\frak s}}} L^\mi G^\mk u_{\mi, \mk}+u,\label{equ-v}\end{equation} where $|u|_{\frak s}<|v|_{\frak s}$ and $u_{\mi, \mk}\in {\rm Ind}_{\frak {hc}^{(0)}}^{\frak {hc}}V$. Note that $\frak {hc}^+u_{\mi, \mk}=0$

Set ${\rm deg}_{\frak s}v=(\bmi, \bmk)$.
If $\bmk\ne\mathbf0$,   acting $F_{\hat{\underline k}-\frac12}$ on \eqref{equ-v},  we get 
\begin{eqnarray}F_{\hat{\underline k}-\frac12}v=v_1+v_2+F_{\hat{\underline k}-\frac12}u,\label{action-F}
\end{eqnarray} 
where 
\begin{eqnarray}
&&v_1=\sum_{\stackrel{(\mi, \mk)\in{\rm supp}_{\frak s}(v)}{ \text w(\mi, \mk)=|v|_{\frak s}}}
L^{\mi} [F_{\hat{\underline k}-\frac12}, G^{\mk}]u_{\mi, \mk},\  {\rm deg}_{\frak s}(v_1)=(\bmi, \mk_1)\\
&&v_2=\sum_{\stackrel{(\mi, \mk)\in{\rm supp}_{\frak s}(v)} { \text w(\mi, \mk)=|v|_{\frak s}} }[F_{\hat{\underline k}-\frac12}, L^{\mi}]G^{\mk}u_{\mi, \mk}, {\rm deg}_{\frak s}(v_2)=(\mi_1, \bmk)\ \text{if}\ v_2\ne0.
\end{eqnarray}

Noticing that $\mi_1\prec\bmi$, $\mk_1\prec \bmk$ and $|F_{\hat{\underline k}-\frac12}u|<|v|_{\frak s}-\hat{\underline k}+\frac12$ or $F_{\hat{\underline k}-\frac12}u=0$ by Lemma \ref{homo}.  
So $F_{\hat{\underline k}-\frac12}v\ne0$ and ${\rm deg}_{\frak s}(F_{\hat{\underline k}-\frac12}v)\prec{\rm deg}_{\frak s}(v)$. It contradicts to the choice of $v$.


So we can suppose that $\bmk=\mathbf0, \bmi\ne0$ and then
\begin{equation}v=\sum_{\stackrel{ (\mi, \mathbf0)\in {\rm supp}_{\frak s}(v)}
{\text w(\mi, \mathbf0)=|v|_{\frak s}}}L^{\mi}u_{\mi, \mathbf 0}+u,\label{equ-u}\end{equation} where $|u|_{\frak s}<|v|_{\frak s}$.

For $\mi=(\cdots, i_2, i_1)\in\mathbb M$, denote by $(L^{\mi})^*:=H_{1}^{i_1}H_{2}^{i_2}\cdots$.  By action of $(L^{\bmi})^*$ on \eqref{equ-u}, we get 
$(L^{\bmi})^*v=au_{\bmi, \mathbf0}$ for some $a\in\mathbb C^*$ since $(L^{\bmi})^*\cdot L^{\mi}u_{\mi, \mathbf0}=0$ if $(\mi, \mathbf0)\prec(\bmi, \mathbf0)$ (see \cite{ZD}), and $(L^{\bmi})^*u=0$ for $u$ in \eqref{equ-u} by Lemma \ref{homo}. This is a contradiction to the choice of $v$. 

So we can suppose that \begin{equation}v=\sum_{(\mj, \ml)\in {\supp}_{\frak {hc}}(v)} H^\mj F^\ml u_{\mj, \ml},\label {equ-v-HF}\end{equation} with a minimal $\text {deg}_{\frak {hc}}(v)$ in all nonzero $v\in W$,  where $u_{\mj, \ml}\in V$.  Note that $\mj=\mathbf{0}$ or $\hat{j}>q$, and 
$\ml=\mathbf{0}$ or $\hat{l}>q$ for any  $(\mi, \mj, \mk, \ml)\in \mathrm{supp}(v)$.

 Set $m_h:={\rm max}\{\hat{\hat{j}}\mid (\mj, \ml)\in {\supp}_{\frak {hc}} (v)\}$ and $m_f:= {\rm max}\{\hat{\hat{l}}\mid (\mj, \ml)\in {\supp}_{\frak {hc}} (v)\}$. They can be $0$ but at least one is positive. Clearly, $m_h, m_f>q$. 
 
If $m_h\ge m_f$,  then by action of $L_{m_h}$ on \eqref{equ-v-HF}  and the condition that $d+(n+1)z\ne0$ for any $n\in\mathbb Z^*$, we get a  contradiction  to the choice of $v$. 

If $m_h<m_f$, then by action of $G_{m_f-\frac12}$ on \eqref{equ-v-HF}  and the condition that $d+(n+1)z\ne0$ for any $n\in\mathbb Z^*$,  we get a similar contradiction.
\end{proof}
%

\begin{rema}{ In Theorem \ref{simple-theo} (resp.  Theorem \ref{simple-theo1}),  the actions of  $L_{q+i}, H_i, G_{q+i-\frac12}, F_{i-\frac12}$ on $\mathrm{Ind}_{\frak g^{(0, -q)}}^{\frak g}V$  for  all $i>t$ (resp. $i>0$)  are locally nilpotent. It follows that $\mathrm{Ind}_{\frak g^{(0, -q)}}^{\frak g}V$ is a simple smooth  ${\frak g}$-module of central charge $(c, z)$ and of level $0$.}

Especially,  if $q=0$ in Theorem \ref{simple-theo1}, then $L_0$ acts on $V$ is not free. So $V$ is one-dimensional, and  ${\rm Ind}_{\frak g^{(0, -q)}}^{\frak g}V$  is a highest weight module. 

\end{rema}

\subsection{Characterizing  simple smooth ${\frak g}$-modules}

In this subsection, we give a characterization of simple smooth ${\frak g}$-modules $S$ of central charge $(c, z)$ at level zero. Moreover, we can suppose that $H_0$ acts on $S$ as a scalar $d$.


\begin{prop}\label{main-lev0}
Let $S$ be a simple $\frak g$-module with $d+(n+1)z\ne0$ for any $n\in\mathbb Z^*$. Then the following statements are equivalent:

\begin{itemize}
  \item[{\rm (1)}]   There exists $t\in\mathbb Z_+$ such that the actions of $L_i, H_i, G_{i-\frac{1}{2}}, F_{i-\frac{1}{2}}$ for all $i\ge t$ on $S$ are locally finite.
 \item[{\rm (2)}]   There exists $t\in\mathbb Z_+$ such that the actions of $L_i, H_i, G_{i-\frac{1}{2}}, F_{i-\frac{1}{2}}$ for all $i\ge t$ on $S$ are locally nilpotent.
 \item[{\rm (3)}] There exist $c, z\in \mathbb C$, $q\in\mathbb N$ and  a simple $\mathfrak g^{(0, -q)}$-module $V$ such that both conditions $(a)$ and $(b)$ in Theorem \ref{simple-theo} or Theorem \ref{simple-theo1} are satisfied  and
$S\cong\mathrm{Ind}_{\frak g^{(0, -q)}}^{\frak g}V$.
\end{itemize}

\end{prop}
\begin{proof}
First we prove $(1)\Rightarrow(3)$. Suppose that $S$ is a simple $\mathfrak g$-module and there exists $t\in\mathbb Z_+$ such that the actions of
$L_i, H_i, G_{i-\frac12}, F_{i-\frac12},  i\ge t$ are locally finite.  Thus we can find nonzero $w\in S$ such that $L_tw = \lambda w$ for some $\lambda\in\mathbb C$.

Similar to the proof of \cite[Theorem 6]{CG} on Pages 81-83, (where the simplicity of the $\frak{hv}$-module $S$ assumed there was not used), we obtain that, 
$$V'=\{v\in S\mid L_iv =H_iv =0, \,\,\mbox{for\ all}\ i>m\}\ne\{0\}$$
for some $m\in \mathbb Z_+$.

Using this just established result we can easily obtain that  \begin{equation}L_{n+i}V^\prime=G_{n+i-\frac12}V^\prime=F_{n+i-\frac12}V^\prime=0\end{equation} for some $n>t$ and any $i>0$.

 For any $r, k\in \mathbb Z$, we consider the following vector space
 $$N_{r, k}=\{v\in S\mid L_iv=G_{i-\frac12}v=H_jv=F_{j-\frac12}v=0, \quad \mbox{for\ all}\ i>r, j>k\}.$$
Clearly  $N_{r, k}\neq0$ for sufficiently large $r, k\in\mathbb Z_+$. 
Moreover $N_{r, k}=0$ if $k<0$ since $H_0v\ne0$ for any $v\in S$.
Thus we can find a smallest nonnegative integer, saying $s$, with $V:=N_{r, s}\neq 0$ for some $r\ge s$. Denote $q=r-s\ge 0$ and $V=N_{s+q, s}$
For any  $i>s+q, j>s$ and $p\geq -q$, it follows from $i+p> s$  that
\begin{eqnarray*}&&L_{i}(H_{p}v)=-pH_{i+p}v=0,\ H_{j}(H_{p}v)=0,\\
&& G_{i-\frac{1}{2}}(H_{p}v)=F_{i+p-\frac12}v=0,\ F_{j-\frac{1}{2}}(H_{p}v)=0
\end{eqnarray*}
for any $v\in V$, respectively.
Clearly, $H_pv\in V$ for
all $p\geq -q$. Similarly, we can also obtain $L_kv, G_{k-\frac12}v, F_{k-q+\frac12}v\in V$
for all $k\in \mathbb N$. Therefore, $V$ is a $\mathfrak g^{(0, -q)}$-module.

\noindent{\bf Case 1. $s\ge 1$.}   By the definition of $V$, we can obtain that the action of $H_{s}$  on $V$ is injective by Theorem \ref{simple-theo}. Since $S$ is simple
and generated by $V$, there exists a canonical surjective map
$$\pi:\mathrm{ Ind}_{c,d, z}V \rightarrow S, \quad \pi(1\otimes v)=v,\quad \forall  v\in V.$$
Next we only need to show that $\pi$ is also injective, that is to say, $\pi$ as the canonical map is bijective.  Let $N=\mathrm{ker}(\pi)$. Obviously, $N\cap V=0$. If $N\neq0$, we can choose a nonzero vector $v\in N\setminus V$ such that $\mathrm{deg}(v)=(\mi,\mj, \mk, \ml)$ is minimal possible.
Note that $N$ is a $\mathfrak g$-submodule of $\mathrm{Ind}_{\frak g^{(0, -q)}}^{\frak g}V$.
By the claim in the proof of Theorem \ref{simple-theo} we can create a new vector $u\in N$  with $\mathrm{deg}(u)\prec(\mi, \mj, \mk,\ml)$, which is a contradiction (where we did not use the assumption on the simplicity of $V$). This forces $K=0$, that is, $S\cong \mathrm{Ind}_{\frak g^{(0, -q)}}^{\frak g}V$ in Theorem \ref{simple-theo}. Then $V$ is a simple $\mathfrak g^{(0, -q)}$-module.

\noindent{\bf Case 2.} $s=0$ and $r\ge0 \,(q=r)$.  Similar to the argument above for $s\ge1$ and using the proof of  Theorem \ref{simple-theo1} and  the assumption that $d+(n+1)z\ne0$ for any $n\in\mathbb Z^*$, we can deduce that $S\cong \mathrm{Ind}_{\frak g^{(0, -q)}}^{\frak g}V$  and $V$ is a simple $\mathfrak g^{(0, -q)}$-module.

Moreover, $(3)\Rightarrow(2)$
 and $(2)\Rightarrow(1)$ are clear. This completes the proof.
\end{proof}

\begin{lemm}\label{RL}
Let $V$ be a simple smooth ${\frak g}$-module. Then there exists $t\in\mathbb Z_+$ such that the actions of $L_i, H_i, G_{i-\frac{1}{2}}, F_{i-\frac{1}{2}}$ for all $i\ge t$ on $V$ are locally nilpotent.
\end{lemm}
\begin{proof} It is clear (also see Lemma 4.2 in \cite{LPX1}).
\end{proof}


From  Theorem \ref{main-lev0} and Lemma \ref{RL}, we are in a position to state one of our main result.

\begin{theo}\label{MT} Let $S$ be a simple smooth $\frak g$-module of central charge $(c, z)$ at level $0$, with $H_0$ acting as a scalar $d$ with $d+(n+1)z\ne0$ for any $n\in\mathbb Z^*$.  Then $S$ is isomorphic to a simple module of the form $\mathrm{Ind}_{\frak g^{(0, -q)}}^{\frak g}V$, where $V$ is a simple $\mathfrak g^{(0, -q)}$-module for some $q\in\mathbb N$ in Theorem \ref{simple-theo} or Theorem \ref{simple-theo1}.\end{theo}

\section{Applications and examples}

In this section, we utilize  Theorems \ref{mainthm} and  \ref{MT} to provide a characterization of simple highest weight ${\frak g}$-modules and simple Whittaker ${\frak g}$-modules. Additionally, we present several examples of simple smooth ${\frak g}$-modules that are also weak (simple) $\mathcal{V}(c,z,\ell)$-modules.

\subsection{Smooth modules at nonzero level}
We first character simple highest weight modules and Whittaker modules over $\frak g$ at nonzero level. 

\subsubsection{Highest weight modules}

For $h, d, c, z\in\mathbb C$, let $\mathbb C v$  be one-dimensional  ${\frak g}^{(0, 0)}$-module  defined by
$ L_0v=hv, H_0v=d v,  {\bf c}_1v=cv, {\bf c}_2v=zv, {\bf c}_3v=\ell v$ and ${\frak g}^{+}$ acts trivially on $v$.
The {Verma module} for $\frak g$ can be defined by
${\rm Ind}_{\frak g^{(0, 0)}}^{\frak g}\mathbb C v.$

\begin{theo}\label{thmmain}
Let $S$ be a ${\frak g}$-module
(not necessarily weight) on which every
element in the algebra ${\mathfrak{g}}^{+}$ acts
locally nilpotently. Then the following statements hold.
\begin{enumerate}
\item[\rm(i)] The module $S$ contains a nonzero vector $v$ such that
${\mathfrak{g}}^{+}\, v=0$.
\item[\rm(ii)] If $S$ is simple at nonzero level, then $S$ is a highest weight module.
\end{enumerate}
\end{theo}

\begin{proof}
(i) It follows from \cite[Theorem 1]{MZ1} that there exists a nonzero vector $v\in S$ such that $L_iv=H_iv=0$ for any $i\in\mathbb Z_+$. 

If  $F_{\frac{1}{2}}v= 0$,  applying $L_i$ we see that $F_{i-\frac{1}{2}}v=0$ $i\in\mathbb Z_+$, i.e., $L_iv=H_iv=F_{i-\frac12}v=0$ for any $i\in\mathbb Z_+$.

If  $F_{\frac{1}{2}}v\ne 0$, then we have $L_1^iF_{\frac12}v=(-1)^iF_{i+\frac12}v$ for any $i\in\mathbb Z_+$.   As $L_1$ acts 
locally nilpotently on $S$, it follows that there exists some $n \in\mathbb Z_+$ such that $F_{j+\frac12}v = 0$ for $j> n$. 
Replace $v$ by $F_{n+\frac12}v$ if $F_{n+\frac12}v\ne 0$, we can get $L_iv=H_iv=0,  F_{n+\frac12}v=0$ for any $i\in\mathbb Z_+$.  Repeating the above steps, we  also find a nonzero vector $v\in S$ such that
\begin{equation}L_iv=H_iv= F_{i-\frac12}v=0,  \ \forall i\in\mathbb Z_+.\label{singular vector}\end{equation}

If  $G_{\frac{3}{2}}v= 0$,   repeatedly applying $L_1$ we see that $G_{i+\frac{1}{2}}v=0$ $i\in\mathbb Z_+$, i.e., $L_iv=H_iv=F_{i-\frac12}v=G_{i+\frac{1}{2}}v=0$ for any $i\in\mathbb Z_+$.

If  $G_{\frac{3}{2}}v\ne 0$, then we have $L_1^iG_{\frac32}v=(-1)^ii!G_{i+\frac32}v$ for any $i\in\mathbb Z_+$.   As $L_1$ acts locally nilpotently on $S$, it follows that there exists some $n \in\mathbb Z_+$ such that $G_{j+\frac12}v = 0$ for $j>n$. Replace $v$ by $G_{n+\frac12}v$ if $G_{n+\frac12}v\ne 0$, we can get $L_iv=H_iv=F_{i-\frac12}v=0,  G_{n+\frac12}v=0$ for any $i\in\mathbb Z_+$.  Repeating the above steps, we  also  find a nonzero $v\in S$ such that $G_{i+\frac12}v=L_iv=H_iv=F_{i-\frac12}v=0$ for any $i\in\mathbb Z_+$.

If  $G_{\frac{1}{2}}v= 0$,  we can see  that   $L_iv=H_iv=F_{i-\frac12}v=G_{i+\frac{1}{2}}v=0$ for any $i\in\mathbb Z_+$. Statement (i) holds.

If  $G_{\frac{1}{2}}v\ne 0$,  let $u=G_{\frac{1}{2}}v$. we can verify that   that   $L_iu=H_iu=F_{i-\frac12}u=G_{i+\frac{1}{2}}u=0$ for any $i\in\mathbb Z_+$. Statement (i) holds also in this case.  

So there exists a nonzero vector $v\in S$ such that $F_{i-\frac{1}{2}}u=0$, $G_{i-\frac{1}{2}}u=0, L_{i}u=0, H_{i}u=0$ for all $i\in\mathbb Z_+$.

(ii) By (i), we know that $S$ is a simple smooth ${\frak g}$-module with $n_S=0$ and $ r_S\le 1$. From Theorem \ref{main-lev0} and Theorem \ref{mainthm} we know that $S\cong U^{\mathfrak{g}}\otimes H(z)^{\mathfrak{g}}$ as $\mathfrak{g}$-modules for some  simple modules $H\in \mathcal{R}_{{\mathfrak {hc}}}, z\in\mathbb C$ and  $U\in \mathcal{R}_{{\frak s}}$.  Moreover,  $H={\rm Ind}^{{\mathfrak {hc}}}_{{\mathfrak {hc}}^{(0)}}(\mathbb C v)$ is a simple highest weight module over ${\frak g}$. Note that every
element in the algebra ${\frak s}^{(1)}$  acts
locally nilpotently on $\mathbb C v\otimes U$ by the assumption. This implies that the same property also holds on $U$. From \cite[Theorem 4.3]{LPX1} we know that $U$ is a simple highest weight ${\frak s}$-module. This completes the proof.
\end{proof}

As a direct consequence of Theorem \ref{thmmain}, we have
\begin{coro}
Let $S$ be a simple smooth ${\frak g}$-module at nonzero level with $r_S\leq 1$ and $n_S=0$. Then $S$ is a highest weight module.
\end{coro}

\begin{proof}
The assumption that $r_S\leq 1$ and $n_S=0$ implies that there exists a nonzero vector $v\in M$ such that ${\mathfrak{g}}^{+}v=0$. Then $M=U({\mathfrak{g}}^{-}+\frak g_0)v$. It follows that each element in ${\mathfrak{g}}^{+}$ acts
locally nilpotently on $M$. Consequently, the desired assertion follows directly from Theorem \ref{thmmain}.
\end{proof}

\subsubsection{Whittaker modules}

Now we will focus on the Whittaker modules over $\frak g$ at nonzero level.

For  $m\in \mathbb{Z}_+$,  
from \cite{LPX1}, we know that any simple finite-dimensional $\mathfrak{g}^{(m, 0)}$-module is one dimensional.
Let $\varphi_m$ be a Lie superalgebra homomorphism $\mathfrak{g}^{(m, 0)}\to \mathbb C$. Then
$$ \varphi_k(L_{2m+1+j})=\varphi_m(H_{m+1+j}) =  \varphi_m(G_{m+j+\frac12})=\varphi_m(F_{j+\frac12})=0, \,\,\forall j\in \mathbb N.$$
Let $\mathbb C w$ be one-dimensional $\mathfrak{g}^{(m, 0)}$-module with $xw=\varphi_m(x)w$ for all $x\in\mathfrak{g}^{(m, 0)}$ and $\varphi({\bf c}_1)=c, \varphi({\bf c}_2)=z$, and $\varphi({\bf c}_3)=\ell$  for some $c, z, \ell\in\mathbb C$. 
The induced $\frak g$-module
\begin{equation}\label{whittaker}\widetilde{W}_{\phi_m}={\rm Ind}_{\frak g^{(m, 0)}}^{\frak g} \mathbb C w_{\phi_m}\end{equation} will be called the universal
Whittaker module with respect to $\phi_m$. And any nonzero quotient
of $\widetilde{W}_{\phi_m}$ will be called a Whittaker module with
respect to $\phi_m$.

For the above $\phi_m$ with $\phi_m({\bf c}_3)=\ell\ne 0$, we define a new Lie superalgebra homomorphism
$\phi'_m: {\frak s}^{(m)}\rightarrow \mathbb C$ as follows.
\begin{eqnarray}\phi'_m({\bf c}_1)&=&\phi_m({\bf c}_1)-\frac32+12\frac{\phi({\bf c}_2)^2}{\phi({\bf c}_3)},\\
\phi'_m(L_{k})&=&0,\ \forall k\ge 2m+1,
\nonumber\\
\phi'_m(L_n)&=&\phi_m(L_n)-\frac{1}{2\ell}\sum_{k=0}^m\phi(H_{k})\phi( H_{-k+n})+\frac{(n+1)c_2}{\ell}\phi(H_n)\nonumber\\
&&+\frac{1}{2\ell}\sum_{k=0}^{m}(k+1)\phi(F_{k+\frac12}) \phi(F_{-k-\frac12+n}), \forall n=m,m+1, \ldots,2m,\\
\phi'_m(G_{k+m+\frac12})&=&0, \forall k\in \mathbb N.\end{eqnarray}
 Then we have  the universal Whittaker
${\frak s}$-module $W_{\phi'_m}:={\rm Ind}_{{\frak s}^{(m)}}^{\frak s} \mathbb C w_{\phi'_m}$, where $x\cdot
w_{\phi'_m}=\phi'_m(x)w_{\phi'_m},\forall x\in {\frak s}^{(m)}$.

\begin{theo}\label{thm-whittaker} Suppose that $m\in \mathbb Z_+$, and $\phi_m$ and $\phi'_m$ are given above with $\phi_m({\bf c}_3)=\ell\ne
0$, and $\phi_m({\bf c}_1)=c, \phi_m({\bf c}_2)=z\in\mathbb C$. Let $H=U({\mathfrak {hc}})w_{\phi_m}$ in $W_{\phi_m}$.  Then we have

\begin{enumerate}
\item[{\rm (1)}] $\widetilde{W}_{\phi_m}\cong W_{\phi'_m}^{\frak g}\otimes H(z)^{\frak g}$.
Consequently, each simple whittaker module with respect to $\phi_m$
is isomorphic to   $T^{\frak g}\otimes H(z)^{\frak g}$ for a simple
quotient $T$ of $W_{\phi'_m}$.
\item[{\rm (2)}]  The  $\frak g$-module $\widetilde{W}_{\phi_m}$ is  simple if and only if
$W_{\phi'_m}$ is a simple ${\frak s}$-module. Consequently,  $\widetilde{W}_{\phi_m}$ is simple if and only if
$(\phi'_m(L_{2m-1}),\phi'_m(L_{2m}))\ne(0,0)$, i.e.,
\begin{eqnarray*}&&2\phi_m(L_{2m})\phi_m({\bf c}_3)+\phi_m(H_{m})^2\ne0,\,\,{\text{or}}\\
&&\phi_m(L_{2m-1})\phi_m({\bf c}_3)+\phi_m(H_{m})\phi_m(H_{m-1})\ne0
\end{eqnarray*} when  $m\ge 2$.
$\widetilde{W}_{\phi_1}$ is simple if and only if
$(\phi'_1(L_{1}),\phi'_1(L_{2}))\ne(0,0)$, i.e.,
\begin{eqnarray*} &&2\phi_1(L_{2})\phi_1({\bf c}_3)+\phi_1(H_{1})^2\ne0,\,\,{\text{or}}\\
&&\phi_1(L_{1})\phi_m({\bf c}_3)+\phi(H_0)\phi_1(H_{1})+2\phi_1({\bf c}_2)\phi_1(H_1)\ne0
\end{eqnarray*}  when $m=1$.

\item[{\rm (3)}] Let $T_1, T_2$ be simple quotients of $W_{\phi'_m}$. Then
$T_1^\frak g\otimes H(z)^{\frak g}\cong T_2^\frak g\otimes H(z)^{\frak g}$
if and only if $T_1\cong T_2$.\end{enumerate}
\end{theo}
 \begin{proof} From Lemma \ref{H-whittaker}, we know that $H$ is a simple  ${\mathfrak {hc}}$-module.

 (1)   
 From simple computations we see that  
 $$H\cong {\rm Ind}_{\frak g^{(m, 0)}}^{\frak g^{(m, -\infty)}} \mathbb C w_{\phi_m}\cong \mathbb C w_{\phi'_m}\otimes H(z)^{\frak g^{(m, -\infty)}}$$ 
 as $\frak g^{(m, -\infty)}$-modules, where the action of $\frak g^{(m, -\infty)}$ on $\mathbb C w_{\phi'_m}$ is given by
 $({\mathfrak {hc}}+\mathbb C {\bf c}_2) w_{\phi'_m}=0, xw_{\phi'_m}=\phi'_m(x)w_{\phi'_m}$ for all $x\in {\frak s}^{(m)}$. Therefore from Lemma \ref{induced} and \eqref{rep2}, we have 
 \begin{eqnarray*}\widetilde{W}_{\phi_m}&\cong& {\rm Ind}_L^{\frak g}({\rm Ind}_{\frak g^{(m, 0)}}^{\frak g^{(m, -\infty)}} \mathbb C
 w_{\phi_m})\cong {\rm Ind}_{\frak g^{(m, -\infty)}}^{\frak g}(\mathbb C w_{\phi'_m}\otimes H(z)^{\frak g^{(m, -\infty)}})\\
 &\cong& {\rm Ind}_{\frak g^{(m, -\infty)}}^{\frak g} \mathbb C w_{\phi'_m}\otimes H(z)^{\frak g}\cong W_{\phi'_m}^{\frak g}\otimes H(z)^{\frak g}.\end{eqnarray*}

Parts (2) and (3) follow from (5.4), \cite[Proposition 5.5]{LPX1}   and some easy
computations.
\end{proof}

The following result characterizes simple Whittaker $\frak g$-modules.

\begin{theo}\label{thmmain17}
Let $M$ be a ${\frak g}$-module at nonzero level
(not necessarily weight) on which ${\mathfrak{g}}^{+}$ acts
locally finitely. Then the following statements hold.
\begin{enumerate}
\item[\rm(i)] The module $M$ contains a nonzero vector $v$ such that
${\mathfrak{g}}^{+}\, v\subseteq{\rm span}_{\mathbb C}\{v, G_{\frac12}v\}$.
\item[\rm(ii)] If $M$ is simple, then $M$ is a Whittaker module or a highest weight module.
\end{enumerate}
\end{theo}
\begin{proof}
From \cite[Theorem 5.9 (i)]{TYZ}, there exists $v\in M$ such that $\frak g_{\bar0}^+v\subset \mathbb C v$. It is clear that $G_{k+\frac12}G_{l+\frac12} v\in\mathbb C v$ and $F_{k+\frac12}F_{l+\frac12}v=0$ for any $k, l\in\mathbb N$. Moreover $U(\frak g^+)v$ is finite dimensional.
Then the theorem follows from \cite[Proposition 3.3]{LPX0} and \cite[Theorem 4.3]{LPX1}.
\end{proof}

\subsection{Smooth modules of level zero}

In this subsection, we give some examples of simple smooth modules of level zero.
\subsubsection{Highest weight modules}
If $\ell=0$ in 6.1.1, then by Theorem \ref{simple-theo1} (also see \cite{AJR}), the Verma module ${\rm Ind}_{\frak g^{(0, 0)}}^{\frak g}\mathbb C v$
 is irreducible if $d+(n+1)z\ne0$ for any $n\in\mathbb Z^*$.

\subsubsection{Whittaker modules}

If $\phi_m({\bf c}_3)=0$ in 6.1.2, then by Theorem \ref{simple-theo1}, we also have
\begin{theo}\label{whittaker-zero}
For $m\in \mathbb Z_+$ and $\phi_m({\bf c}_3)=0$, the Whittaker module $\widetilde{W}_{\phi_m}$ is simple if and only if $\varphi_m(H_{m})\ne 0$.
\end{theo}
\begin{proof}

Let
$$
V_{\varphi_m}={\rm Ind}^{\mathfrak g^{(0, 0)}}_{\mathfrak g^{(m, 0)}}\mathbb C w.
$$
Similar to the proof of Theorem \ref{simple-theo1}, it is tedious but straightforward to check that $V_{\varphi_m}$ is a simple $\mathfrak g^{(0, 0)}$-module  if  $\varphi_m(H_{m})\ne 0$. From Theorem \ref{simple-theo}, we obtain the corresponding induced  $\frak g$-module ${\rm Ind}_{\frak g^{(0, 0)}}^{\frak g}V_{\varphi_k}$ is simple  if $\varphi_m(H_{m})\ne 0$. Clearly ${\rm Ind}_{\frak g^{(0, 0)}}^{\frak g}V_{\varphi_k}$ is exactly the Whittaker module $W({\varphi_m,{c, z}})$. Moreover if $\varphi(H_m)=0$, then the Whittaker module $\widetilde{W}_{\phi_m}$ has a proper submodule generated by $  F_{-\frac12}w $.
\end{proof}

\subsection{Smooth modules that are not tensor product modules}

For characterizing simple ${\frak g}$-module which are not tensor product modules, we need the following lemma.

\begin{lemm}\label{smooth submodule}
Let $M=U^{{\frak g}}\otimes V^{{\frak g}}$ be a simple smooth ${\frak g}$-module with $n_M>1$ and nonzero level, where $U\in{\mathcal R}_{{\frak s}}$ and $V\in{\mathcal R}_{{\mathfrak {hc}}}$ are simple. Set  $V_0={\rm Ann}_V({\mathfrak {hc}}^{(n_M)})$ and $M_0= {\rm Ann}_M({\mathfrak {hc}}^{(n_M)})$. Then  $V_0$ is a simple ${\frak g}^{(0,-n_M+1)}$-module, and
$M_0=U\otimes V_0$. Hence $M_0$ contains a simple $ {\mathfrak {hc}}^{(-n_M+1)}$-submodule.
\end{lemm}
\begin{proof} For any nonzero $u\in U$, $u\otimes V_0$ is a simple $ {\mathfrak {hc}}^{(-n_M+1)}$-submodule of $M_0$.
\end{proof}

Lemma \ref{smooth submodule}  means that if $M\in{\mathcal R}_{{\frak g}}$ is not a tensor product module, then $M_0$ contains no simple ${\mathfrak {hc}}^{(-n_M+1)}$-submodule.

Here we will consider the case $n_M=2$.
Let $\mathfrak{b}=\mathbb C h+\mathbb C e$ be the 2-dimensional solvable Lie algebra with basis $h,e$ and subject to Lie bracket $[h,e]=e$.
The following concrete example using \cite[Example 13]{LMZ} and \cite[Example 7.4]{TYZ}  tells us
 how to construct induced smooth ${\frak g}$-modules from a $\mathbb C[e]$-torsion-free simple $\mathfrak{b}$-module.

\begin{exam}\label{ex-7.4}
Let $W=(t-1)^{-1}\mathbb C[t,t^{-1}]\oplus (t-1)^{-1}\mathbb C [t,t^{-1}]\theta$ be the  associative superalgebra with $t, t^{-1}\in W_{\bar 0},$ $ \theta\in W_{\bar 1}$. Note that $\theta^2=0$. From \cite[Example 13]{LMZ} we know that  $W$ is a direct sum of two isomorphic simple $\mathfrak{b}$-module  whose structure  is given by
\begin{equation*}
h\cdot f(t, \theta)=t\frac{d}{dt}(f(t, \theta))+\frac{f(t, \theta)}{t^2(t-1)},\,\,
e\cdot f(t, \theta)=tf(t, \theta),\forall f(t, \theta)\in W.
\end{equation*}

For $c,z,z'\in\mathbb C, \ell\in\mathbb C^*$, we also can make $W$ into a ${\frak g}^{(0,0)}$-module by
$$\begin{aligned}&L_0\cdot f(t, \theta)=h\cdot f(t, \theta), \,H_{1}\cdot f(t, \theta)=e\cdot f(t, \theta), \\
&H_0\cdot f(t, \theta)=z'f(t, \theta), H_{1+i}\cdot f(t, \theta)=L_i\cdot f(t, \theta)=0,\,\,\,i\in\mathbb Z_+,\\
&F_{\frac12}\cdot f(t, \theta)=\theta t f(t, \theta),  G_{\frac12}f(t, \theta)=\frac\partial{\partial \theta}f(t, \theta),\\
&F_{\frac12+i}\cdot f(t, \theta)=G_{i+\frac12}\cdot f(t, \theta)=0,\,\,\,i\in\mathbb Z_+,\\
& {\bf c}_1\cdot f(t, \theta)=cf(t, \theta),\, {\bf c}_2\cdot f(t, \theta)=zf(t, \theta),\, {\bf c}_3\cdot f(t, \theta)=\ell f(t, \theta),\\
\end{aligned}
$$where $f(t, \theta)\in W$.
Then $W$ is a simple ${\frak g}^{(0,0)}$-module. Clearly, the action of $H_{1}$ on $W$ implies that  $W$ contains no simple ${\mathfrak {hc}}^{(0)}$-module. Then $M_0={\rm Ind}_{{\frak g}^{(0,0)}}^{{\frak g}^{(0,-1)}}W$ is a simple ${\frak g}^{(0,-1)}$-module and contains no simple ${\mathfrak {hc}}^{(-1)}$-module.  Let $M={\rm Ind}_{{\frak g}^{(0,-1)}}^{{\frak g}}M_0$. It is easy to see $n_M=2$ and $r_M= 3$. The proof of  Proposition \ref{prop-r2-n2} implies that $M$ is a simple smooth $\frak g$-module. Lemma \ref{smooth submodule}  means that  $M$ is not a tensor product $\frak g$-module.

\end{exam}

 \begin{rema}According to Theorem \ref{mainthm}, if $n_M$ equals 0 or 1, then simple smooth ${\frak g}$-modules must be tensor product modules. However, Example \ref{ex-7.4} demonstrates that for any $n_M>1$, there exist simple smooth ${\frak g}$-modules that are not tensor product modules.
 On the other hand, Theorem \ref{thm-whittaker} states that the smooth module $\widetilde{W}_{\phi_m}$ is a tensor product of smooth modules over ${\mathfrak {hc}}$ and $\mathfrak s$, even when $n_{\widetilde{W}_{\phi_m}}>1$ and $r_{\widetilde{W}_{\phi_m}}>1$.
\end{rema}

\begin{center}{\bf Acknowledgments}
\end{center}

The authors gratefully acknowledge the partial financial support  from  the NNSF (11971315, 12071405) and NSERC (311907-2020). Part of the research in this paper was carried out during K.Z. visited Shanghai University in the period of time Mar.19-May 22, 2023.


\begin{thebibliography}{ACDC}


\bibitem{ALZ} D. Adamovi{\'c}, R. L{\"u}, K. Zhao, Whittaker modules for the affine Lie algebra $A^{(1)}_1$,
Adv. Math. 289 (2016), 438-479.



\bibitem{AJR0} D. Adamovi{\'c}, B. Jandri{\'c}, G. Radobolja, On the N = 1 super Heisenberg-Virasoro
vertex algebra, to appear in Lie Groups, Number Theory, and Vertex Aglebras, Proceedings of the Conference “Representation Theory XVI", June 24-29, IUC- Dubrovnik, Contemporary Mathematics, 167-178.

\bibitem{AJR}D. Adamovi{\'c}, B. Jandri{\'c}, G. Radobolja,  The $N=1$ super Heisenberg-Virasoro vertex algebra at level zero, J. Algebra Appl.  21(12), (2022) 2350003.








\bibitem{ACKP} E. Arbarello, C. De Concini, V. G. Kac, C. Procesi, Moduli spaces
of curves and representation theory, Comm. Math. Phys. 117 (1988),
1-36.

\bibitem{AP} D. Arnal, G. Pinczon, On algebraically irreducible representations of the Lie algebra $\mathfrak{sl}(2)$, J. Math. Phys. 15 (1974), 350-359.

\bibitem{A} A. Astashkevich, On the structure of Verma modules over Virasoro and Neveu-
Schwarz algebras, Comm. Math. Phys. 186 (1997), 531-562.


\bibitem{B} A. A. Balinsky, Classification of the Virasoro, the Neveu–Schwarz and the Ramond-type simple Lie superalgebras, Funct. Anal. Appl. 21 (1987) 308–309.



\bibitem{BM} P. Batra, V. Mazorchuk, Blocks and modules for Whittaker pairs, J. Pure Appl. Algebra,
215 (2011), 1552-1568.

\bibitem{BBFK} V. Bekkert, G. Benkart, V. Futorny, I. Kashuba,  New irreducible modules for Heisenberg and affine Lie algebras, J. Algebra 373 (2013), 284-298.



\bibitem{BO} G. Benkart, M. Ondrus, Whittaker modules for generalized Weyl
algebras, Represent. Theory. 13 (2009), 141-164.






\bibitem{C} H. Chen, Simple restricted modules over the  $N=1$ Ramond algebra
  as weak modules   for vertex operator superalgebras, J. Algebra, 621 (2023)   41-57.


\bibitem{CG} H. Chen, X. Guo, New simple modules for the Heisenberg-Virasoro algebra, J. Algebra, 390 (2013), 77-86.






\bibitem{CHL}B. Chen,  P. Hao,  Y. Liu, Supersymmetric warped conformal field theory. Physical Review D, 102   (2020), 065016.


\bibitem{Chris} K. Christodoupoulou, Whittaker modules for Heisenberg
algebras and imaginary Whittaker modules for affine Lie algebras,
 J. Algebra. 320 (2008), 2871-2890.

\bibitem{DMS} P. Di Francesco, P. Mathieu, D. S\'{e}n\'{e}chal, Conformal field theory, Springer-Verlag, New York, 1997.



\bibitem{FHL} I. Frenkel, Y. Huang, J. Lepowsky, On axiomatic approaches to vertex operator algebras and modules, Mem. Am. Math. Soc., vol. 104, 1993.

\bibitem{FLM} I. Frenkel, J. Lepowsky, A. Meurman, Vertex operator algebras and the monster, Pure and Appl. Math. 134 (1998).



\bibitem{G} D. Gao,   Simple restricted modules for the Heisenberg-Virasoro algebra, J. Algebra 574 (2021), 233-251.

\bibitem{GSW} M. Green, J. Schwarz,  E. Witten,   Superstring Theory. New York: Cambridge University Press  (1987).

\bibitem{Gu} P. Guha,  Integrable geodesic flows on the (super)extension of the Bott-Virasoro group. Lett. Math. Phys. 52   (2000),311-328.

\bibitem{GSWX}H. Guo,  J. Shen,  S. Wang, K. Xu, Beltrami algebra and symmetry of the Beltrami equation on Riemann surfaces. J. Math. Phys. 31 (1990),  2543-2547.









\bibitem{HY} J. T. Hartwig, N. Yu, Simple Whittaker modules over
free bosonic orbifold vertex operator algebras,  Proc. Amer. Math. Soc. 147 (2019), no. 8, 3259-3272.







\bibitem{IK1}  K. Iohara, Y. Koga, Representation theory of Neveu-Schwarz and Remond algebras I: Verma modules. Adv. Math. 177(2003), 61-69.

\bibitem{IK2} K. Iohara, Y. Koga,   Representation theory of Neveu-Schwarz and Remond algebras II: Fock modules. Ann. Inst. Fourier 53 (2003), 1755-18


\bibitem{KT} V. G. Kac, I. T. Todorov, Superconformal current algebras and their unitary representations, Comm. Math. Phys., 102 (1985), 337-347.


\bibitem{KW}V. Kac, W. Wang, Vertex operator superalgebras and representations,in: Mathematical Aspects of Conformal and Topological Field Theories and Quantum Groups, Contemporary Math., 175 Amer. Math. Soc., Providence,  1994, 161-191.



\bibitem{KW1} V. Kac, M. Wakimoto, Unitarizable highest weight representation of the Virasoro, Neveu-Schwarz and Ramond algebras, Lecture Notes in Physics, Vol. 261, Springer, Berlin, 1986.

\bibitem{Ko} B. Kostant, On Whittaker vectors and representation theory,
Invent. Math. 48 (1978), 101-184.



\bibitem{LL} J. Lepowsky and  H. Li, Introduction to vertex operator algebras and their representations,Progress in Math., Vol. 227, Birkh\" auser, Boston, 2004.

\bibitem{Li} H. Li, Local systems of vertex operators, vertex superalgebras and modules, J. Pure Appl. Algebra 109 (1996), 143-195.









\bibitem{LPX0} D. Liu, Y. Pei, L. Xia, Whittaker modules for the super-Virasoro algebras, J. Algebra Appl. 18
(2019), 1950211.


\bibitem{LPX1}  D. Liu, Y. Pei, L. Xia, Simple restricted modules for Neveu-Schwarz algebra, J. Algebra, 546 (2020), 341-356.









\bibitem{LMZ} R. L{\"u},  V. Mazorchuk and K. Zhao, Classification of simple weight modules over the 1-Spatial ageing algebra, Algebr. Represent. Theory 18 (2015), 381-395.



\bibitem{LvZ}  R. L{\"u}, K. Zhao, Generalized oscillator representations of the twisted Heisenberg-Virasoro algebra, Algebr. Represent. Theory 23 (2020), 1417-1442.



\bibitem{Ma} P. Marcel,  Extensions of the Neveu-Schwarz Lie superalgebra. Commun. Math. Phys. 207  (1999), 291-306.

\bibitem{MOR} P. Marcel,  V. Ovsienko, C. Roger,   Extension of the Virasoro and Neveu-Schwarz algebras and generalized Sturm-Liouville operators. Lett. Math. Phys. 40  (1997), 31-39.


\bibitem{MZ1} V.~Mazorchuk, K.~Zhao, Characterization of simple highest weight modules, Canad. Math. Bull. 56 (2013), no. 3, 606-614.

\bibitem{MZ2} V.~Mazorchuk, K.~Zhao, Simple Virasoro modules which are locally finite
over a positive part, Selecta Math. (N.S.),  20  (2014),  no. 3, 839-854.

\bibitem{MD1} E. McDowell, On modules induced from Whittaker modules, J. Algebra 96 (1985), 161-177.

\bibitem{MD2} E. McDowell, A module induced from a Whittaker module, Proc. Amer. Math. Soc. 118 (1993), 349-354.






\bibitem{PB} Y. Pei, C. Bai,  Balinsky-Novikov superalgebras and infinite-dimensional Lie superalgebras, J. Alg. Appl. Vol. 11, No. 6 (2012) 1250119 (29 pages)




\bibitem{TYZ} H. Tan, Y. Yao, K. Zhao, Simple restricted modules over the Heisenberg-Virasoro algebra as VOA modules, arxiv:2110. 05714v2. 


\bibitem{XL} Y. Xue, R. L\"{u},  Simple weight modules with finite dimensional weight spaces over Witt superalgebras.  J. Algebra. 574 (2021), 92-116.


\bibitem{XZ} Y. Xue, K. Zhao, Simple representations of the Fermion-Virasoro algebras, preprint, June of 2022.


\bibitem{ZD} W. Zhang, C. Dong, W-algebra $W(2,2)$ and the vertex operator algebra $L(1,0)\otimes L(1,0)$, Commun.
Math. Phys. 285 (2009), 991-1004.


\end{thebibliography}
\end{document}